\documentclass[12pt, reqno]{amsart}

\allowdisplaybreaks
\usepackage{amssymb,amsthm,amsmath,bm}
\usepackage[numbers,sort&compress]{natbib}

\title[2D Micropolar Equations]{On the initial- and boundary-value problem for 2D micropolar equations with only angular velocity dissipation}
\author{Quansen Jiu, \, Jitao Liu, \, Jiahong Wu, \, Huan Yu}
\address[Quansen Jiu]{School of Mathematical Sciences, Capital Normal Universit
y, Beijing, 100048, P.R. China.}
\email{jiuqs@mail.cnu.edu.cn}
\address[Jitao Liu]{College of Applied Sciences, Beijing University of Technology, Beijing, 100124, P. R. China.}
\email{jtliu@bjut.edu.cn,\,\,\,jtliumath@qq.com}
\address[Jiahong Wu]{Department of Mathematics, Oklahoma State University, Stillwater, OK 74078, USA.}
\email{jiahong.wu@okstate.edu}
\address[Huan Yu] {Institute of Applied Physics and Computational Mathematics, Beijing, 100088, P.R. China.}
\email{yuhuandreamer@163.com}


\theoremstyle{plain}
\newtheorem{corollary}{Corollary}[section]
\newtheorem{theorem}{Theorem}[section]
\newtheorem{lemma}{Lemma}[section]
\newtheorem{proposition}{Proposition}[section]

\theoremstyle{definition}
\newtheorem{definition}{Definition}[section]

\let\f=\frac
\let\p=\partial
\def\R{\mathbb R}

\newcommand{\beq}{\begin{equation}}
\newcommand{\eeq}{\end{equation}}
\newcommand{\ben}{\begin{eqnarray}}
\newcommand{\een}{\end{eqnarray}}
\newcommand{\beno}{\begin{eqnarray*}}
\newcommand{\eeno}{\end{eqnarray*}}

\begin{document}

\begin{abstract}
This paper focuses on the initial- and boundary-value problem
for the two-dimensional micropolar equations with only angular
velocity dissipation in a smooth bounded domain. The aim
here is to establish the global existence and uniqueness of
solutions by imposing natural boundary conditions and minimal
regularity assumptions on the initial data. Besides, the global
solution is shown to possess higher regularity when the
initial datum is more regular. To obtain these results, we overcome
two main difficulties, one due to the lack of full dissipation and
one due to the boundary conditions. In addition to the global
regularity problem, we also examine the large-time behavior of
solutions and obtain explicit decay rates.
\end{abstract}
\maketitle

\section{Introduction and main results}
\label{intro}
\setcounter{equation}{0}

This paper studies the global existence and uniqueness, and large-time behavior of solutions to the two-dimensional (2D) micropolar equations
in a bounded domain $D$ with smooth boundary. The 2D micropolar equations are a special case of the 3D micropolar equations. The standard 3D incompressible micropolar equations are given by
\begin{equation}\label{3D}
\left\{\begin{array}{ll}
u_t-(\nu+\kappa)\Delta u + u\cdot\nabla u +\nabla\pi=2\kappa\nabla\times w,\\
w_t-\gamma\Delta w+4\kappa w-\mu\nabla\nabla \cdot w+u\cdot\nabla w=2\kappa\nabla\times u,\\
\nabla\cdot u=0,
\end{array}\right.
\end{equation}
where $u=u(x,t)$ denotes the fluid velocity, $\pi(x,t)$ the scalar pressure, $w(x,t)$ the micro-rotation field, and the parameter $\nu$ represents the Newtonian kinematic viscosity, $\kappa$ the micro-rotation viscosity, $\gamma$ and $\mu$ the angular viscosities.

\vskip .1in
In the special case when
$$
u=(u_1(x_1,x_2, t), u_2(x_1, x_2, t), 0),\,\pi = \pi(x_1, x_2, t),\,w = (0, 0, w_3(x_1, x_2, t)),
$$
the 3D micropolar equations reduce to the 2D micropolar equations,
\begin{equation}\label{2D}
\left\{\begin{array}{ll}
u_t-(\nu+\kappa)\Delta u+u\cdot\nabla u+\nabla\pi=2\kappa\nabla^{\perp} w,\\
w_t-\gamma\Delta w+4\kappa w+u\cdot\nabla w=2\kappa\nabla\times u,\\
\nabla\cdot u=0.
\end{array}\right.
\end{equation}
Here $u = (u_1, u_2)$ is a 2D vector with the corresponding scalar vorticity $\Omega$ given by
$$\Omega \equiv \nabla\times u=\p_1{u_2}-\p_2{u_1},$$
while $\omega$ represents $\omega_3(x_1,x_2,t)$ for simplicity, which is a scalar function with
$$
{\nabla^{\perp}}w =(\p_2w, -\p_1w).
$$

\vskip .1in
The micropolar equations were introduced in 1965 by C.A. Eringen
to model micropolar fluids (see, e.g., \cite{Er}). Micropolar fluids
are fluids with microstructure. Certain anisotropic fluids, e.g. liquid crystals which are made up of dumbbell
molecules, are of this type. They belong to a class of non-Newtonian fluids with nonsymmetric stress tensor (called polar fluids) and include, as a special case, the classical fluids modeled by the Navier-Stokes equations. In fact,  when micro-rotation effects are neglected, namely $w = 0$, (\ref{3D}) reduces to the incompressible Navier-Stokes equations. The micropolar equations are significant generalizations of the Navier-Stokes equations and cover many more phenomena such as fluids consisting of particles suspended in a viscous medium. The micropolar equations have been extensively applied and studied by many engineers and physicists.

\vskip .1in
Due to their physical applications and mathematical significance,
the well-posedness problem and large-time decay issue on the
micropolar equations have attracted considerable attention recently
from the community of mathematical fluids \cite{CP,DC,Er2,GLuk}.
Lukaszewicz in his monograph \cite{GLuk} studied the well-posedness
problem on the 3D stationary as well as the time-dependent
micropolar equations. In \cite{DC2}, Dong and Chen obtained the
global existence and uniqueness, and sharp algebraic time decay
rates for the 2D micropolar equations (\ref{2D}).

\vskip .1in
More recent efforts are focused on the 2D micropolar equation with
partial dissipation, which naturally bridge the inviscid micropolar
equation and the micropolar equation with full dissipation. The
global regularity problem for the inviscid equation is currently
out of reach. In \cite{DZ} Dong and Zhang examined (\ref{2D}) with the micro-rotation viscosity $\gamma=0$ and established the
global regularity by
observing that the combined quantity
$$\Omega-\f{2\kappa}{\nu+\kappa}w$$
obeys a transport-diffusion equation. Another partial dissipation case, \eqref{2D} with $\nu= 0,\, \gamma>0,\, \kappa> 0$ and $\kappa\neq\gamma$, was examined by Xue, who was able to obtain the global well-posedness in the frame work of Besov spaces \cite{LX}. We remark that the requirement $\kappa\neq\gamma$ in \cite{LX} is not crucial and it is not
difficult to see that the global well-posedness remains valid even when $\kappa=\gamma$. Recently Dong, Li and Wu took on the case when \eqref{2D} involves only the angular viscosity
dissipation \cite{DLW}. They
proved the global (in time) regularity by fully exploiting the structure of system and controling the vorticity via the evolution of a combined quantity of the vorticity and the micro-rotation angular velocity. In addition, \cite{DLW} introduced a diagonalization process to
eliminate the linear terms in order to obtain the large-time behavior of
the solutions.

\vskip .1in
Most of the results we mentioned above are for the whole
space $\mathbb{R}^2$ or $\mathbb{R}^3$. In many real-world
applications, the flows are often restricted to bounded domains
with suitable constraints imposed on the boundaries and these
applications naturally lead to the studies of the initial-
and boundary-value
problems. In addition, solutions of the initial- and boundary-value problems may exhibit much richer phenomena than those of the whole
space counterparts. This paper is devoted to the initial- and
boundary-value problem for the 2D micropolar equations with only angular viscosity dissipation,
\begin{equation}\label{eq1}
\left\{\begin{array}{ll}
u_t+u\cdot\nabla u+\nabla\pi=2\kappa{\nabla^{\perp}}w,\\
w_t-\gamma\Delta w+4\kappa w+u\cdot \nabla w=2\kappa \nabla\times u,\\
\nabla\cdot u=0,
\end{array}\right.
\end{equation}
with the natural boundary condition
\begin{equation}\label{eq2}
u\cdot{\bm n}|_{\p D}=w|_{\p D}=0
\end{equation}
and the initial condition
\begin{equation}\label{eq20}
(u,w)(x,0)=(u_0,w_0)(x),\quad\,\hbox{in}\,\,D,\\
\end{equation}
where $D\subset \R^2$ represents a bounded domain with smooth boundary and $\bm{n}$ is the unit outward normal vector. In addition, we also impose the following compatibility conditions
\begin{equation}\label{eq3}
\left\{\begin{array}{ll}
u_0\cdot{\bm n}|_{\p D}=w_0|_{\p D}=0,&\\
\nabla\cdot u_0=0,&\\
u_0\cdot\nabla u_0+\nabla\pi_0=2\kappa{\nabla^{\perp}}w_0,\quad&\hbox{on}\,\,\p D,\\
4\kappa w_0+u_0\cdot \nabla w_0=\gamma\Delta w_0+2\kappa \nabla\times u_0,\quad&\hbox{on}\,\,\p D,
\end{array}\right.
\end{equation}
where $\pi_0$ is determined by the divergence-free condition $\nabla\cdot u_0=0$ with the Neumann boundary condition
\begin{equation}\label{eq4}
\nabla \pi_0\cdot{\bm n}|_{\p D}=[2\kappa{\nabla^{\perp}}w_0-u_0\cdot\nabla u_0]\cdot{\bm n}|_{\p D}.
\end{equation}

\vskip .1in
Our first goal here is to establish the global existence and
uniqueness of solutions to \eqref{eq1}-\eqref{eq20} by imposing the
least regularity assumptions on the initial data. We assume here that the
initial vorticity $\Omega_0 =\nabla\times u_0$ is in the Yudovich
class and $w_0\in H^2(D)$, and obtain the following result.

\begin{theorem}\label{T1}
Let $D\subset\R^2$ be a bounded domain with smooth boundary. Assume
$(u_0, w_0)$ satisfies
$$
u_0\in H^1(D), \quad \Omega_0\in L^\infty(D), \quad w_0\in H^2(D)
$$
and the compatibility conditions (\ref{eq3}) and (\ref{eq4}). Then \eqref{eq1}-\eqref{eq20} has a unique global smooth solution $(u,w)$
satisfying
\ben
u\in L^\infty(0,T;H^1(D)),\quad \Omega\in L^\infty(0,T; L^\infty(D)),\quad  w\in L^\infty(0,T;H^2(D)) \label{regclass}
\een
for any $T>0$.
\end{theorem}

\vskip .1in
We further establish higher regularity of the solution stated in
Theorem \ref{T1}. By imposing the higher regularity
$$
u_0 \in H^3 (D), \quad w_0 \in H^4(D),
$$
we obtain that the corresponding solution $(u, w)$ remains this class
for all time. More precisely, we have the following theorem.

\begin{theorem}\label{T2}
In addition to the conditions in Theorem \ref{T1}, we further assume
$$
u_0 \in H^3 (D), \quad w_0 \in H^4(D),
$$
then the solution remains the same regularity for all time, namely
\beno
u\in L^\infty(0,T;H^{3}(D)),\quad\quad w\in L^\infty(0,T;H^4(D)),
\eeno
for any $T>0$.
\end{theorem}

\vskip .1in
The proof of Theorem \ref{T1} is divided into several major steps.
The first step is to establish the global $H^1$-bound for $(u, w)$.
We make use of the equation of the vorticity $\Omega=\nabla\times u$,
\begin{equation}\label{vor}
\partial_t \Omega + u\cdot\nabla \Omega = -2 \kappa \Delta w
\end{equation}
and the equation of $\nabla w$, and bound the vortex stretching term in the equation of $\nabla w$
suitably. The second step combines the global $H^1$-bound obtained in the first step and Schauder's
fixed point theorem to prove the global existence of weak solutions.
The third step establishes the global boundedness of the vorticity
in $L^\infty$. Due to the presence of the bad term $-\Delta w$,
(\ref{vor}) itself does not allow us to extract the desired global
bound. To overcome this difficulty, we consider the
combined quantity
$$
Z=\Omega+\frac{2\kappa}{\gamma}w,
$$
which satisfies
\begin{equation}\label{zeq}
\partial_t Z+u\cdot\nabla Z-
\frac{4\kappa^2}{\gamma}Z+\frac{8\kappa^2}{\gamma}(1+\frac\kappa\gamma)w=0.
\end{equation}
This equation eliminates the difficulty and yields the desired bound. The
fourth step is to prove the global $H^2$-bound for $w$. Due to the
no-slip boundary condition for $w$, the standard energy estimates
would not work. Instead we obtain a global bound for
$\|\partial_t w\|_{L^2(D)}$ first and the global bound for $\|w\|_{H^2(D)}$
follows as a consequence.

\vskip .1in
To prove the higher regularity bounds specified in Theorem \ref{T2}, we further exploit the equation of $Z$, namely (\ref{zeq}). By taking the gradient of (\ref{zeq}) and making use of a logarithmic Sobolev type
inequality, we obtain the global bounds for
$$
\|\nabla \Omega\|_{L^\infty(0, T; L^q(D))} \quad\mbox{and}\quad
\|\nabla u\|_{L^\infty(0, T; L^\infty(D))},
$$
where $q\in(1,\infty)$. These bounds allow us to further consecutively
establish the bounds for
$$
\|\Delta w\|_{L^q(0,T; L^q(D))}, \quad
\|\nabla\partial_t w\|_{L^\infty(0, T; L^2(D))} \quad\mbox{and}\quad
\|w\|_{L^\infty(0, T; H^4(D))}.
$$
We remark that, in contrast to the whole space case,
the estimates here are not so straightforward and involve more steps.

\vskip .1in
This paper also looks into the large-time behavior of solutions of
(\ref{eq1}) with an extra velocity damping term and wihtout the $4 \kappa w$ in the $w$ equation, namely
\begin{equation}\label{eq5}
\left\{\begin{array}{ll}
u_t+\kappa u+u\cdot\nabla u+\nabla\pi=2\kappa{\nabla^{\perp}}w,\\
w_t-\gamma\Delta w+u\cdot \nabla w=2\kappa \nabla\times u,\\
\nabla\cdot u=0.
\end{array}\right.
\end{equation}
It is clear that Theorem \ref{T1} and Theorem \ref{T2} remain valid for (\ref{eq5}). The damping term in (\ref{eq5}) and the elimination of $4 \kappa w$ term does not affect the
regularity results. Our last focus is the large-time behavior and
we show that, if
\begin{equation}\label{gaka}
\gamma> 4\kappa,
\end{equation}
then the $H^1$-norm of $(u, w)$ decays exponentially in time,
$$
\|(u(t), w(t)\|_{H^1(D)} \le C \, e^{-\widetilde{C} t},
$$
where $C>0$ is a constant depending on the $H^1$-norm of $(u_0, w_0)$ and
$\widetilde{C}>0$ depends on $\gamma$ and $\kappa$ only.
More precisely, we have the following theorem.

\begin{theorem}\label{T3}
Let $D\subset\R^2$ be a bounded domain with smooth boundary. Assume
the conditions on $(u_0,w_0)$, as stated in Theorem \ref{T1}. If we further assume
\begin{equation}\label{eq6}
\gamma>4\kappa,
\end{equation}
then the solution $(u, w)$ of (\ref{eq5}) satisfies
\begin{equation}\label{eq7}
\|u\|_{H^1(D)}^2+\|w\|_{H^1(D)}^2\leq C \, e^{-\widetilde{C} t},
\end{equation}
where $C>0$ is a constant depending on the $H^1$-norm of $(u_0, w_0)$ and
$\widetilde{C}>0$ depends on $\gamma$ and $\kappa$ only.
\end{theorem}

The condition in (\ref{gaka}) is necessary and sharp. As pointed
out in \cite{DLW}, when (\ref{gaka}) is violated, the solution may
grow in time.

\vskip .1in
The rest of this paper is divided into five sections. The second
section serves as a preparation and presents a list of facts
and tools for bounded domains
such as embedding inequalities and logarithmic type interpolation
inequalities.  Section \ref{globalweak} establishes the global existence
of $H^1$-weak solutions, one major step in the proof of Theorem \ref{T1}.
Section \ref{Yudovich} proves the global $L^\infty$ bound for $\Omega$
and the global $H^2$ bound for $w$. This step completes the
proof of Theorem \ref{T1}. Section \ref{higher} proves Theorem \ref{T2}, the higher global regularity bounds. The last section is devoted to the large-time
behavior result stated in Theorem \ref{T3}.

\vskip .3in
\section{Preliminaries}
\label{prel}
\setcounter{equation}{0}

\vskip .1in
This section serves as a preparation. We list a few basic tools
for bounded domains to be used in the subsequent sections. In particular,
we provide the Gagliardo-Nirenberg type inequalities, the logarithmic
type interpolation inequalities and regularization estimates for elliptic and parabolic equations in bounded domains. These estimates will also be
handy for future studies on PDEs in bounded domains.

\vskip .1in
We start with the well-known Gagliardo-Nirenberg inequality for bounded
domains (see, e.g., \cite {NIR}).
\begin{lemma}\label{P1}
Let $D\subset\R^n$ be a bounded domain with smooth boundary. Let
$1\leq p, q, r \leq\infty$ be real numbers and $j\le m$ be non-negative
integers. If a real number $\alpha$ satisfies
$$\frac{1}{p} - \frac{j}{n} = \alpha\,\left( \frac{1}{r} - \frac{m}{n} \right) + (1 - \alpha)\frac{1}{q}, \qquad \frac{j}{m} \leq \alpha \leq 1,
$$
then
$$\| \mathrm{D}^{j} f \|_{L^{p}(D)} \leq C_{1} \| \mathrm{D}^{m} f \|_{L^{r}(D)}^{\alpha} \| f \|_{L^{q}(D)}^{1 - \alpha} + C_{2} \|f\|_{L^{s}(D)},$$
where $s > 0$, and the constants $C_1$ and $C_2$ depend upon $D$ and the indices $p,q,r,m,j,s$ only.
\end{lemma}

\vskip .1in
Especially, the following special cases will be used.

\begin{corollary}\label{C1}
Suppose $D \subset\R^2$ be a bounded domain with smooth boundary. Then

(1)\,\,\,$\| f \|_{L^{4}(D)} \leq C\, (\|  f \|_{L^{2}(D)}^{\f12} \| \nabla f \|_{L^{2}(D)}^{\f12} + \| f \|_{L^{2}(D)}),\,\,\,\forall f\in H^1(D);$\vskip 0.2cm

(2)\,\,\,$\| \nabla f \|_{L^{4}(D)} \leq C\, (\|  f \|_{L^{2}(D)}^{\f14} \| \nabla^2 f \|_{L^{2}(D)}^{\f34} + \| f \|_{L^{2}(D)}),\,\,\,\forall f\in H^2(D);$\vskip 0.2cm

(3)\,\,\,$\| f \|_{L^{\infty}(D)} \leq C\, (\|  f \|_{L^{2}(D)}^{\f12} \| \nabla^2 f \|_{L^{2}(D)}^{\f12} + \| f \|_{L^{2}(D)}),\,\,\,\forall f\in H^2(D);$\vskip 0.2cm

(4)\,\,\,$\| f \|_{L^{\infty}(D)} \leq C\, (\|  f \|_{L^{2}(D)}^{\f23} \| \nabla^3 f \|_{L^{2}(D)}^{\f13} + \| f \|_{L^{2}(D)}),\,\,\,\forall f\in H^3(D).$
\end{corollary}

\vskip .1in
The lemma below provides estimates for products and commutators (see, e.g., \cite{MB}).

\begin{lemma}\label{commutator-est}
Let $D\subset \R^2$ be a bounded domain with smooth boundary.
Then for any multi-indices $\alpha$ and $\beta$,
\begin{equation}\label{c-0}
 \|D^\alpha(fg)\|_{L^2(D)}\leq C(\|f\|_{L^\infty(D)}\|g\|_{H^{|\alpha|}(D)}+\|f\|_{H^{|\alpha|}(D)}\|g\|_{L^\infty(D)})
\end{equation} for some constant $C$ depending on $D$ and $\alpha,$ and
\begin{equation}\label{c-1}\|D^\beta(fg)-fD^\beta g\|_{L^2(D)}\leq C(\|\nabla f\|_{L^\infty(D)}\|g\|_{H^{|\beta|-1}(D)}+\|f\|_{H^{|\beta|}(D)}\|g\|_{L^\infty(D)})
\end{equation}for some constant $C$ depending on $D$ and $\beta.$
\end{lemma}

\vskip .1in
We will need the Poincar\'{e} type inequality (see, e.g.,
\cite{evans}).

\begin{lemma}\label{poincare}
Let $D\subset\R^2$ be a bounded domain with smooth boundary and $u\in H_0^{1}(D)$. Then there exists a constant $C$ depending on $D$ only such that
$$\|u\|_{L^{2}(D)}\leq C\,\|\nabla u\|_{L^{2}(D)}.$$
\end{lemma}

\vskip .1in
It is well-known that the standard singular integral operators are bounded on $L^q(\mathbb{R}^n)$ for any $q\in(1,\infty)$ (see, e.g., \cite{St}).
The lemma below provides the bounded domain version of this fact (see,
e.g., \cite{Bour,FT,St}). More precisely, this lemma allows us to control
the gradient of a vector field in $L^q$ in terms of its curl.

\begin{lemma}\label{curl}
Let $D\subset\R^2$ be a bounded domain with smooth boundary and let $m$ be a positive integer. If a vector field $u\in L^q(D)$ with $q\in(1,\infty)$ satisfies
$$
\nabla\times u\in W^{m-1,q}(D), \qquad \nabla\cdot u=0, \qquad u\cdot{\bm n}|_{\p D}=0,
$$
then there is a constant $K=K(D, m)>0$ (independent of $q$) such that
$$
\|u\|_{W^{m,q}(D)}\leq K\,q\,(\|\nabla\times u\|_{W^{m-1,q}(D)}+\|u\|_{L^{q}(D)}).
$$
\end{lemma}

\vskip .1in
The lemma next presents a logarithmic interpolation inequality for vector
fields  defined on bounded domains. It is taken from \cite{KS}.

\begin{lemma}\label{log-inequ}
Let $D\subset\R^2$ be a bounded domain with smooth boundary. Then, for any $0<\alpha<1$,
$$
\|\nabla u\|_{L^{\infty}(D)}\leq C\|\nabla\times u\|_{L^\infty(D)}(1+{\rm log}\,(e+ \|\nabla\times u\|_{C^{\alpha}(D)})).
$$
\end{lemma}

\vskip .1in
The next two lemmas state the regularization
estimates for elliptic and parabolic equations
defined on bounded domains (see, e.g., \cite{evans,GT,Lady}).

\begin{lemma}\label{elliptic}
Let $D\subset\R^2$ be a bounded domain with smooth boundary. Consider the elliptic boundary-value problem
\begin{equation}\label{eel}
\left\{\begin{array}{ll}
-\Delta f=g\quad&\hbox{in}\,\,D,\\
f=0\quad&\hbox{on}\,\,\p{D}.
\end{array}\right.
\end{equation}
If, for $p\in(1,\infty)$ and an integer $m\geq -1$,  $g\in W^{m,p}(D)$,  then (\ref{eel}) has a unique solution $f$ satisfying
$$
\|f\|_{W^{m+2,p}(D)}\leq C\|g\|_{W^{m,p}(D)},
$$
where $C$ depending only on $D, \,m$ and $p$.
\end{lemma}

\begin{lemma}\label{heat}
Let $D\subset\R^2$ be a bounded domain with smooth boundary. Assume, for $q\in [2,\infty)$,
$$
\varphi\in W_0^{1,q}(D),\qquad g\in L^q(0,T;L^q(D)).
$$
Assume $f\in L^2(0,T;H^1_0(D))$ with $f_t\in L^2(0,T;H^{-1}(D))$ is a weak solution of the parabolic system
\beno
\left\{\begin{array}{ll}
f_t-\Delta f=g,\quad&\hbox{in}\,\,{D},\\
f=0\quad&\hbox{on}\,\,\p{D},\\
f|_{t=0}=\varphi\quad&\hbox{in}\,\,{D}.
\end{array}\right.
\eeno
Then, there exists a constant $C$ depending only on $q,\,D$,  such that
$$\|f_t\|_{L^q(0,T;L^{q}(D))}+\|\Delta f\|_{L^q(0,T;L^{q}(D))}\leq C(\|g\|_{L^q(0,T;L^{q}(D))}+\|\varphi\|_{W_0^{1,q}(D)}).$$
\end{lemma}

\vskip .1in
We also recall Kato's well-posedness result on the
2D Euler equation (see \cite{Kato}).

\begin{lemma}\label{Euler}
Let $D\subset\R^2$ be a bounded domain with smooth boundary.
Consider the initial- and boundary-value problem
\beno
\left\{\begin{array}{ll}
u_t-u\cdot\nabla u+\nabla p=f,\\
\nabla\cdot u=0,\\
u|_{t=0}=u_0(x),~~u\cdot \mathbf{n}|_{\partial D}=0.
\end{array}\right.
\eeno
Assume $u_{0}\in C^{1+\gamma}(\overline{D})$ with $0<\gamma<1$ satisfying $\nabla\cdot u_{0}(x)=0$ and $u_{0}(x)\cdot \mathbf{n}|_{\partial D}=0.$ Let $T>0$ and $f\in C([0,T];C^{1+\gamma}(\overline{D}))$. Then there exists a unique solution $(u,p)$ such that $(u,p)\in C^{1+\gamma}(\overline{D}\times[0,T])$.
\end{lemma}

%

\vskip .3in

\section{Global existence of $H^1$-weak solutions}
\label{globalweak}
\setcounter{equation}{0}

This section proves the global existence of $H^1$-weak solutions of
(\ref{eq1})-(\ref{eq20}). This result is an important step in
the proof of Theorem \ref{T1}. To be more precise, we first
provide the definition of weak solutions of (\ref{eq1})-(\ref{eq20})
and then state the main result of this section as a proposition.

\begin{definition}\label{weak}
Let $D\subset\R^2$ be a bounded domain with smooth boundary.
Assume $(u_0,w_0)\in H^1(D)$. A pair of measurable functions $(u, w)$ is called a weak solution of (\ref{eq1})-(\ref{eq20}) if
\beno
(1)&&u\in C(0,T;H^1(D)),\,\,w\in C(0,T;L^2(D))\cap L^2(0,T;H_0^1(D));\\
(2)&&\int_{D}u_0\cdot\varphi_0dx+\int_0^T\int_{D}\big[u\cdot\varphi_t+u\cdot\nabla\varphi\cdot u+2\kappa \nabla^{\perp}w\cdot\varphi\big]dxdt=0,\\
&&\int_0^T\int_{D}\big[w\psi_t+\gamma\nabla w\cdot\nabla\psi+4\kappa w\psi+u\cdot\nabla\psi\cdot w-2\kappa \nabla\times u\psi\big]dxdt\\
&&=-\int_{D}w_0\cdot\psi_0dx;
\eeno
holds for any $(\varphi,\,\psi)\in C^\infty([0,T]\times D)$ with $\nabla\cdot\varphi=\varphi|_{\p D}=\varphi(x,T)=0$ and $\psi\cdot\bm{n}|_{\p D}=\psi(x,T)=0.$
\end{definition}

\vskip .1in
The main result of this section is stated in the following proposition.

\begin{proposition}\label{weak2}
Let $D\subset\R^2$ be a bounded domain with smooth boundary. Assume $(u_0,w_0)\in H^{1}(D)$. Then (\ref{eq1})-(\ref{eq20}) has a global
weak solution.
\end{proposition}

\vskip .1in
The proof of this proposition relies on the following global $H^1$-bound.

\begin{lemma}\label{L31}
Under the assumptions of  Proposition \ref{weak2}, for any $T>0$, there exists a constant $C$ depending only on $T$ and the initial data
such that
\beno
&&\|u\|_{L^\infty(0,T;H^1(D))}+\|w\|_{L^\infty(0,T;H^1(D))}
+\|w\|_{L^2(0,T;H^2(D))}\leq\, C.
\eeno
\end{lemma}

\begin{proof}[Proof of Lemma \ref{L31}]
We start with the  global $L^2$-bound. Taking the inner product of $\eqref{eq1}$ with $(u, w)$  yields
\beno
&&\f12\f{d}{dt}(\|u\|_{L^2(D)}^2+\|w\|_{L^2(D)}^2)+\gamma\|\nabla w\|_{L^2(D)}^2+4\kappa\|w\|_{L^2(D)}^2\\
&=&2\kappa\int_{D}\nabla^{\perp}w\cdot u\,dx+2\kappa\int_{D}\nabla\times uwdx.
\eeno
Noticing that $\nabla\times u = \partial_{1} u_2-\partial_{2} u_1$
and $\nabla^{\perp}w =(\partial_{2} w, -\partial_{1} w)$, we have
$$
\nabla\times u w =(\partial_{1} u_2-\partial_{2} u_1)w
= \partial_{1}(u_2\, w) - \partial_{2} (u_1\,w)+ \,\nabla^{\perp}w\cdot u.
$$
Integrating by parts and applying the boundary condition for $w$, we have
\ben
&& 2\kappa\int_{D}\nabla^{\perp}w\cdot u\,dx+ 2\kappa\int_{D}\nabla\times uw\,dx  \notag\\
&=&4\kappa\,\int_{D}\nabla^{\perp}w\cdot u\,dx +  2\kappa\int_{\p D}u\cdot{{\bm n}^\perp}wds  \notag\\
&=&4\kappa\,\int_{D}\nabla^{\perp}w\cdot u\,dx  \label{bbb}\\
&\le& \frac{\gamma}{2}\, \|\nabla w\|_{L^2(D)}^2 + C\, \|u\|_{L^2(D)}^2, \notag
\een
where $\mathbf{n}^\perp=(-n_2,n_1)$.
It then follows, after integration in time, that
\ben
&&\|u\|_{L^2(D)}^2+\|w\|_{L^2(D)}^2+\gamma\,\int_{0}^{t}\|\nabla w\|_{L^2(D)}^2\,d\tau+8\kappa\,\int_{0}^{t}\|w\|_{L^2(D)}^2d\tau
\notag\\
&\leq& e^{C\, t}(\|u_0\|_{L^2(D)}^2+\|w_0\|_{L^2(D)}^2) \equiv A_1(t,
\|(u_0, w_0)\|_{L^2}),
\label{A1}
\een
where $C=C(\gamma,\kappa)$. To obtain the global $H^1$-bound for $(u,w)$, we invoke the vorticity
equation
\ben\label{eq31}
\Omega_t+u\cdot\nabla \Omega=-2\kappa\Delta w,
\een
Multiplying (\ref{eq31}) by $\Omega$ and the equation of
$w$ in $(\ref{eq1})$ by $\Delta w$, and applying the boundary
condition $w|_{\p D}=0$, the Cauchy-Schwarz inequality and
Corollary {\ref{C1}}, we have
\ben
&&\f12\f{d}{dt}(\|\Omega\|_{L^2(D)}^2+\|\nabla w\|_{L^2(D)}^2)+\gamma\|\Delta w\|_{L^2(D)}^2+4\kappa\|\nabla w\|_{L^2(D)}^2 \notag\\
&=&\int_{D}u\cdot\nabla w\,{\Delta}w\,dx
-4\kappa\int_{D}\Delta w\,\Omega\, dx \notag\\
&\leq&\|\Delta w\|_{L^2(D)}\|u\|_{L^4(D)}\|\nabla w\|_{L^4(D)}+4\kappa\|\Omega\|_{L^2(D)}\|\Delta w\|_{L^2(D)} \notag\\
&\leq&C\|\Delta w\|_{L^2(D)}^{\f32}\|\Omega\|_{L^2(D)}^{\f12}\|\nabla w\|_{L^2(D)}^{\f12}+C\|\Delta w\|_{L^2(D)}\|\Omega\|_{L^2(D)}^{\f12}\|\nabla w\|_{L^2(D)} \notag\\
&&+4\kappa\|\Omega\|_{L^2(D)}\|\Delta w\|_{L^2(D)} \notag\\
&\leq&\f{\gamma}2\|\Delta w\|_{L^2(D)}^{2}+C\|\nabla w\|_{L^2(D)}^{2}(1+\|\Omega\|_{L^2(D)}^{2})+C\|\Omega\|_{L^2(D)}^2.
\label{A102}
\een
Gronwall's inequality and \eqref{A1} then yield
the following global $H^1$-bound
\ben
&&\|\nabla u(t)\|_{L^2(D)}^2+\|\nabla w(t)\|_{L^2(D)}^2
+\gamma \,\int_0^t\|\Delta w\|_{L^2(D)}^2 \,d\tau + 4\kappa\,\int_0^t \|\nabla w\|_{L^2(D)}^2\,d\tau \notag \\
&\le& C_1\,e^{C_2\, e^{C_3 t}}\,\|(\nabla u_0, \nabla w_0)\|^2_{L^2(D)} \equiv A_2(t), \label{A2}
\een
where $C_1=C_1(\gamma,\kappa)$, $C_2=C_2(\gamma,\kappa,\|(u_0, w_0)\|_{L^2(D)})$ and
$C_3=C_3(\gamma,\kappa)$. This completes the proof of Lemma \ref{L31}.
\end{proof}

\vskip .1in
We now prove Proposition \ref{weak2}.

\begin{proof}[Proof of Proposition \ref{weak2}]
The proof is a consequence of Schauder's fixed point theorem.
We shall only provide the sketches.

\vskip .1in
To define the functional setting, we fix $T>0$ and $R_0$ to
be specified later. For notational convenience, we
write
$$
X\equiv C(0,T; H^1_0(D))\cap L^2(0,T; H^2(D))
$$
with $\|g\|_X\equiv \|g\|_{C(0,T; H^1_0(D))}+\|g\|_{L^2(0,T; H^2(D))}$,
and define
$$
B=\{g\in X\,|\,\|g\|_X\leq R_0\}.
$$
Clearly, $B\subset X$ is closed and convex.

\vskip .1in
We fix $\epsilon\in(0, 1)$ and define a continuous map on $B$. For any $g\in B$, we regularize it and the initial data $(u_0, w_0)$ via the
standard mollifying process,
$$
g^{\epsilon}= \rho^{\epsilon}\ast g, \quad u_{0}^{\epsilon} = \rho^{\epsilon}\ast u_0, \quad w_{0}^{\epsilon} = \rho^{\epsilon}\ast w_0,
$$
where $\rho^{\epsilon}$ is the standard mollifier. According to
Lemma \ref{Euler}, the 2D incompressible Euler equations with smooth external forcing $2\kappa\nabla^{\perp}g^{\epsilon}$ and smooth initial data $u_0^{\epsilon}$
\begin{equation}\label{weq1}
\left\{\begin{array}{ll}
u_t+u\cdot\nabla u+\nabla\pi=2\kappa{\nabla^{\perp}}g^{\epsilon},\\
\nabla\cdot u=0,\\
u(x,0)=u_0^{\epsilon}(x),\quad u\cdot{\bm n}|_{\p D}=0,
\end{array}\right.
\end{equation}
has a unique solution $u^{\epsilon}$. We then solve the linear parabolic equation with the smooth initial data $w_0^{\epsilon}$
\begin{equation}\label{weq2}
\left\{\begin{array}{ll}
w_t-\gamma\Delta w+4\kappa w+u^{\epsilon}\cdot \nabla w=2\kappa \nabla\times u^{\epsilon},\\
w(x,0)=w_0^{\epsilon}(x),\quad w|_{\p D}=0,
\end{array}\right.
\end{equation}
and denote the solution by $w^{\epsilon}$. This process
allows us to define the map
$$
F^{\epsilon}(g)=w^{\epsilon}.
$$
We then apply Schauder's fixed point theorem to construct a sequence of approximate solutions to \eqref{eq1}-\eqref{eq20}. It suffices to show that, for any fixed $\epsilon\in(0, 1)$, $F^{\epsilon}: B\rightarrow B$
is continuous and compact. More precisely, we need to show
\begin{enumerate}
\item[(a)] $\|w^{\epsilon}\|_{B}\leq R_0$;
\item[(b)] $\|F^{\epsilon}(g_1)-F^{\epsilon}(g_2)\|_{B}\leq C\|g_1-g_2\|_{B}$ for $C$ indepedent of $\epsilon$ and any $g_1,\,g_2\in B$.
\end{enumerate}
We verify (a) first. A simple $L^2$-estimate on (\ref{weq1}) leads to
\beno
\|u^\epsilon(t)\|_{L^2(D)} &\le& \|u_0^\epsilon\|_{L^2(D)} + 2\kappa\,\int_0^t \|{\nabla}g^{\epsilon}\|_{L^2(D)}\,d\tau, \\
&\le& \|u_0\|_{L^2(D)} + 2\kappa\,\int_0^t \|{\nabla}g\|_{L^2(D)}\,d\tau.
\eeno
Similar to (\ref{A1}), we have
$$
\|w^\epsilon\|^2_{L^2(D)} + \gamma \,\int_0^t
\|\nabla w^\epsilon\|^2_{L^2(D)}\,d\tau
+ 4 \kappa \int_0^t \|w^\epsilon\|^2_{L^2(D)}\,d\tau
\le \|w_0\|^2_{L^2(D)} + \frac{2\kappa^2}{\gamma}\,\int_0^t \|u^\epsilon\|^2_{L^2(D)}\,d\tau.
$$
To bound the $H^1$-norms, we rely on the equation of $\Omega^\epsilon$,
$$
\Omega^\epsilon_t + u^\epsilon\cdot\nabla \Omega^\epsilon=-2\kappa\Delta g^\epsilon.
$$
A simple energy estimate then yields
$$
\|\Omega^\epsilon\|_{L^2(D)} \le \|\Omega_0\|_{L^2(D)}
+ 2 \kappa\,\int_0^t \|\Delta g\|_{L^2(D)} \,d\tau.
$$
As in (\ref{A102}), we have
$$
\frac{d}{dt} \|\nabla w^\epsilon\|_{L^2(D)}^2
+ \gamma \|\Delta w^\epsilon\|^2_{L^2(D)}\,
+ 4 \kappa \|\nabla w^\epsilon\|^2_{L^2(D)}\,
\le (1+ \|\Omega^\epsilon\|_{L^2(D)}^2)\,(1 + \|\nabla w^\epsilon\|_{L^2(D)}^2).
$$
By Gronwall's inequality,
\beno
&& \|\nabla w^\epsilon\|_{L^2(D)}^2 +\,\gamma \,\int_0^t
\|\Delta w^\epsilon\|^2_{L^2(D)}\,d\tau
+ 4 \kappa \int_0^t \|\nabla w^\epsilon\|^2_{L^2(D)}\,d\tau\\
&& \quad\le (1+\|\nabla w_0\|_{L^2(D)}^2) \int_0^t (1+ \|\Omega^\epsilon\|_{L^2(D)}^2)\,e^{C\,\int_0^\tau (1+ \|\Omega^\epsilon\|_{L^2(D)}^2)\,ds}\,d\tau.
\eeno
Combining the estimates yields
\beno
&& \|w^\epsilon\|^2_{H^1(D)} + \gamma \,\int_0^t \|w^\epsilon\|^2_{H^2(D)}\,d\tau
\le \|w_0\|^2_{L^2(D)}  + \frac{2\kappa^2}{\gamma}\,\int_0^t \|u^\epsilon\|^2_{L^2(D)}\,d\tau\\
&&\qquad + (1+\|\nabla w_0\|_{L^2(D)}^2) \int_0^t (1+ \|\Omega^\epsilon\|_{L^2(D)}^2)\,e^{C\,\int_0^\tau (1+ \|\Omega^\epsilon\|_{L^2(D)}^2)\,ds}\,d\tau.
\eeno
In order for $F^\epsilon$ to map $B$ to $B$, it suffices for the
right-hand side to be bounded by $R_0$. Invoking the bounds for $\|u^\epsilon\|_{L^2}$ and $\|\Omega^\epsilon\|_{L^2}$, we obtain a condition for $T$ and $R_0$,
\ben
&& \|w_0\|_{L^2(D)}^2 + C T\,(1+ \|u_0\|_{L^2(D)}^2 + \|\nabla w_0\|_{L^2(D)}^2
T R^2_0) \notag \\
&& \qquad  +  C T\,(1 + \|\Omega_0\|_{L^2(D)}^2 + T \, R^2_0)\exp(C(\|\Omega_0\|_{L^2(D)}^2 + T \, R^2_0))) \le R_0,
\label{TR0}
\een
where the constants $C$ depend only on the parameters
$\kappa$ and $\gamma$. It is not difficult to see that,
if $T$ is sufficiently small, (\ref{TR0}) would hold.
Similarly, we can show (b) under the condition that $T$
is sufficiently small. Schauder's fixed point theorem then allows
us to conclude that the existence of a solution on
a finite time interval $T$. These uniform estimates
would allow us to pass the limit to obtain a
weak solution $(u,w)$.

\vskip .1in
We remark that the local solution obtained by Schauder's fixed
point theorem can be easily extended into a global solution via
Picard type extension theorem due to the global bounds obtained in
(\ref{A1}) and (\ref{A2}). This allows us to obtain the desired
global weak solution. This completes the proof.
\end{proof}

\vskip .3in
\section{Yudovich regularity and proof of Theorem \ref{T1}}
\label{Yudovich}
\setcounter{equation}{0}

\vskip .1in
The goal of this section is to complete the proof of Theorem \ref{T1}.
To do so, we first establish the Yudovich type regularity for the
vorticity $\Omega$, namely $\Omega\in L^\infty$ for all time and then show that $w\in {L^\infty(0,T; H^2(D))}$. The regularity obtained here for
$\Omega$ and $w$ allows us to prove the uniqueness of the weak solutions
established in the previous section.

\vskip .1in
Recall that $\Omega$ satisfies
\begin{equation}\label{vv}
\Omega_t+u\cdot\nabla \Omega=-2\kappa\Delta w.
\end{equation}
Due to the bad term $-\Delta w$ on the right-hand side, this equation
itself does not allow us to extract a global bound on $\Omega$. To bypass this difficulty, we combine (\ref{vv}) with the equation of $w$,
$$
w_t-\gamma\Delta w+4\kappa w-2\kappa \nabla\times u+u\cdot \nabla w=0
$$
to eliminate the bad term. More precisely, we consider the combined
quantity
$$
Z=\Omega+\frac{2\kappa}{\gamma}w
$$
and the equation that it satisfies
\begin{equation}\label{Z-equ}
	\partial_tZ+u\cdot\nabla Z-\frac{4\kappa^2}{\gamma}Z+\frac{8\kappa^2}{\gamma}(1+\frac\kappa\gamma)w=0,
	\end{equation}
which leads us to the desired global bound. More precisely, we have the
following proposition.

\begin{proposition}\label{w-w2p}
Assume that $(u_0, w_0)$ satisfies the conditions stated in Theorem \ref{T1}.  Let $(u,w)$ be the global weak solution obtained in Proposition
\ref{weak2}. Then the corresponding vorticity $\Omega$ obeys the global bound, for any $2\leq p<\infty$, and any $T>0$ and $0<t\le T$,
\begin{equation*}
\|\Omega\|_{L^\infty(0,T; L^p(D))}\leq C,
\end{equation*}
where the constant $C$ depends only on $D,T$ and the initial data.
\end{proposition}

\begin{proof} We start with the equation of $Z$, namely (\ref{Z-equ}).
For any $2\leq p<\infty$, multiplying \eqref{Z-equ} with $|Z|^{p-2}Z$ and
	integrating on $D$, we obtain
	\begin{equation*}	\frac1p\frac{d}{dt}\|Z\|_{L^p(D)}^p\leq\frac{4\kappa^2}{\gamma}\|Z\|_{L^p(D)}^p+
	\frac{8\kappa^2}{\gamma}(1+\frac\kappa\gamma)
	\|w\|_{L^p(D)}\|Z\|_{L^p(D)}^{p-1},
	\end{equation*}
	i.e.,
	\begin{equation*}
	\begin{split}
	\frac{d}{dt}\|Z\|_{L^p(D)}&\leq\frac{4\kappa^2}{\gamma}\|Z\|_{L^p(D)}
	+\frac{8\kappa^2}{\gamma}(1+\frac\kappa\gamma)
	\|w\|_{L^p(D)},
	\end{split}
	\end{equation*}
	which, according to Gronwall's inequality, implies
	\begin{equation*}
	\|Z\|_{L^p(D)}\leq e^{\frac{4\kappa^2}{\gamma}t}\left(\|Z_0\|_{L^p(D)}
+\,C\,\int_0^t\|w(\tau)\|_{L^p(D)}d\tau\right).
	\end{equation*}
	Noting that $C$ is independent of $p$, we obtain, by letting $p\rightarrow \infty$,
\begin{equation*}\begin{split}
	\|Z\|_{L^\infty(D)}&\leq e^{\frac{4\kappa^2}{\gamma}t}(\|Z_0\|_{L^\infty(D)}
+ C\, \int_0^t\|w(\tau)\|_{L^\infty(D)}d\tau)
\\&\leq e^{\frac{4\kappa^2}{\gamma}t}(\|Z_0\|_{L^\infty(D)}
+C\, \int_0^t\|w(\tau)\|_{H^2(D)}d\tau).
	\end{split}
	\end{equation*}
    Thus, by noticing that $\|w\|_{L^2(0,T;H^2(D))}\leq\ C$ from Lemma \ref{L31}, it is clear that
	\begin{equation}\label{Z--infty}
	\|Z\|_{L^\infty(0,T; L^p(D))}\leq C,
	\end{equation}
	for any $2\leq p\leq\infty$. By the definition of $Z$, Sobolev embedding and Lemma \ref{L31}, we have
	\begin{equation*}\begin{split}
	\|\Omega\|_{L^\infty(0,T; L^p(D))}&\leq (\|Z\|_{L^\infty(0,T; L^p(D))}+\|w\|_{L^\infty(0,T; L^p(D))})\\&\leq C(\|Z\|_{L^\infty(0,T; L^p(D))}+\|w\|_{L^\infty(0,T; H^1(D))})\\&\leq C.
	\end{split}
	\end{equation*}
for any $2\leq p<\infty$. This completes the proof of Proposition \ref{w-w2p}.
\end{proof}

\vskip .1in
Next we prove the global bound for $\|w\|_{H^2(D)}$.  In contrast to the
whole space case, we need to estimate the time derivatives of $(u, w)$
in order to obtain the desired bound.

\begin{proposition}\label{w-H2-est}
	Assume that $(u_0, w_0)$ satisfies the conditions stated in Theorem \ref{T1}.  Let $(u,w)$ be the global weak solution obtained in Proposition
\ref{weak2}. Then, for any $T>0$ and $0<t<T$,
	\begin{equation*}
	\|u_t\|_{L^\infty(0,T; L^2(D))}+\|w_t\|_{L^\infty(0,T; L^2(D))} + \|w\|_{L^\infty(0,T; H^2(D))}\leq C,
	\end{equation*}
where the constant $C$ depends only on $D, T$ and the initial data.
\end{proposition}

\begin{proof}
We first estimate $\|u_t\|_{L^2(D)}$. Differentiating the equation of $u$ in $(\ref{eq1})$ with respect to $t$ and then dotting it with $u_t$, we have
\begin{equation*}\begin{split}
\|u_t\|_{L^2(D)}^2&=-\int_{D} u\cdot\nabla u\cdot u_t\,dx
\,+\,2\kappa \int_{D}  \nabla^{\bot}w\cdot u_t\, dx
\\
&\leq
\frac12\|u_t\|_{L^2(D)}^2+\, C\,\|u\cdot\nabla u\|_{L^2(D)}^2+\kappa\|\nabla^{\bot}w\|_{L^2(D)}^2\\&\leq
\frac12\|u_t\|_{L^2(D)}^2+\, C\,\|u\|_{L^4(D)}^2\|\nabla u\|_{L^4(D)}^2+ \kappa\|\nabla w\|_{L^2(D)}^2\\
&\leq\frac12\|u_t\|_{L^2(D)}^2+\,C\,\|u\|_{H^1(D)}^2(\|u\|_{L^2(D)}^2+\|\Omega\|_{L^4(D)}^2)+\kappa\|\nabla w\|_{L^2(D)}^2.
\end{split}
\end{equation*}
The global bounds in Lemma \ref{L31} and Proposition \ref{w-w2p} then imply
\begin{equation}\label{ut-est}
\|u_t\|_{L^\infty(0, T; L^2(D))}\leq C.
\end{equation}
To estimate $\|w_t\|_{L^2(D)}$, we take the temporal derivative of the $w$-equation in (\ref{eq1}) to get
\begin{equation}\label{wt-equ}
w_{tt}+u\cdot\nabla w_t+u_t\cdot\nabla w+4\kappa w_t=\gamma \Delta w_t+2\kappa\Omega_t.
\end{equation}
Multiplying \eqref{wt-equ} with $w_t$ and integrating on $D$, it follows that
\begin{equation}\label{wt-est}\begin{split}
&\frac12\frac{d}{dt}\|w_t\|_{L^2(D)}^2+\gamma\|\nabla w_t\|_{L^2(D)}^2+4\kappa\|w_t\|_{L^2(D)}^2\\
=&-\int_{D}  w_tu_t\cdot\nabla w \,dx+2\kappa\int_{D}  \Omega_tw_t \,dx.
\end{split}
\end{equation}
By integration by parts, H\"{o}lder's inequality and Young's inequality,
\begin{equation*}\begin{split}
-\int_{D}  w_tu_t\cdot\nabla w dx&=\int_{D}  wu_t\cdot\nabla w_t \,dx\\&\leq
\frac\gamma4\|\nabla w_t\|_{L^2(D)}^2+\frac C\gamma \|u_t\|_{L^2(D)}^2\|w\|_{L^\infty(D)}^2\\&\leq
\frac\gamma4\|\nabla w_t\|_{L^2(D)}^2+\frac C\gamma \|u_t\|_{L^2(D)}^2\|w\|_{H^2(D)}^2.
\end{split}
\end{equation*}
To estimate the second term in the right of \eqref{wt-est}, we make use of the vorticity equation \eqref{vv} and integrate by parts on $D$ to get
\begin{equation*}\begin{split}
2\kappa\int_{D}\Omega_tw_t dx&=-2\kappa\int_{D}u\cdot\nabla \Omega w_t dx-4{\kappa}^2\int_{D}\Delta w w_t dx\\
&=2\kappa\int_{D}  u\cdot\nabla w_t\Omega dx+4{\kappa}^2\int_{D}  \nabla w \cdot\nabla w_t dx,
\end{split}
\end{equation*}
which yields,
\begin{equation*}
2\kappa\int_{D}  \Omega_tw_t dx\leq\frac\gamma4\|\nabla w_t\|_{L^2(D)}^2+\frac {8\kappa^2}\gamma\|\Omega\|_{L^4(D)}^2\|u\|_{H^1(D)}^2+\frac{32\kappa^4}{\gamma}\|\nabla w\|_{L^2(D)}^2.
\end{equation*}
Combining the two estimates above with \eqref{wt-est}, applying \eqref{ut-est}, and invoking Lemma \ref{L31} and Proposition \ref{w-w2p}, we conclude
\begin{equation*}
\|w_t\|_{L^\infty(0,T; L^2(D))}\leq C.
\end{equation*}
By Lemma \ref{elliptic},
\begin{equation}\label{w-H2}\begin{split}
\|w\|_{H^2(D)}&\leq C\left(\|w_t\|_{L^2(D)}+\|u\cdot\nabla w\|_{L^2(D)}+\|w\|_{L^2(D)}+\|\Omega\|_{L^2(D)}\right).
\end{split}
\end{equation}
By H\"{o}lder's inequality, Sobolev embedding inequality and Lemma \ref{curl}, we have
\begin{equation*}\begin{split}
\|u\cdot\nabla w\|_{L^2(D)}&\leq\|u\|_{L^\infty(D)}\|\nabla w\|_{L^2(D)}\\&\leq C\|u\|_{W^{1,p}(D)}\|\nabla w\|_{L^2(D)}\\&\leq C(\|\Omega\|_{L^p(D)}+\|u\|_{H^1(D)})\|\nabla w\|_{L^2(D)},
\end{split}
\end{equation*}
where $2\leq p<\infty$. Together with Lemma \ref{L31} and Proposition \ref{w-w2p}, we obtain the desired global bound. This completes the proof of Proposition \ref{w-H2-est}.
\end{proof}

\vskip .1in
The global bound for $\|\Omega\|_{L^\infty(D)}$ follows as an easy consequence of Proposition \ref{w-H2-est}.

\begin{proposition}\label{omegainfty}
	Assume that $(u_0, w_0)$ satisfies the conditions stated in Theorem \ref{T1}.  Let $(u,w)$ be the global weak solution obtained in Proposition
	\ref{weak2}. Then for any $T>0$ and $0<t\le T$,
	\begin{equation*}
		\|\Omega\|_{L^\infty(0,T; L^\infty(D))}\leq C,
	\end{equation*}
	where the constant $C$ depends only on $D,T$ and the initial data.
\end{proposition}

\begin{proof} According to \eqref{Z--infty},
	\begin{equation}\notag
		\|Z\|_{L^\infty(0,T; L^\infty(D))}\leq C,
	\end{equation}
which implies, by the definition of $Z$, Sobolev's embedding and Proposition \ref{w-H2-est},
	\begin{equation*}\begin{split}
			\|\Omega\|_{L^\infty(0,T; L^\infty(D))}&\leq (\|Z\|_{L^\infty(0,T; L^\infty(D))}+\|w\|_{L^\infty(0,T; L^\infty(D))})\\&\leq C(\|Z\|_{L^\infty(0,T; L^\infty(D))}+\|w\|_{L^\infty(0,T; H^2(D))})\\&\leq C.
		\end{split}
	\end{equation*}
	This completes the proof of Proposition \ref{omegainfty}.
\end{proof}

\vskip .1in
We are now ready to prove Theorem \ref{T1}.

\begin{proof}[Proof of Theorem \ref{T1}]
 In view of Propositions \ref{weak2}, \ref{w-w2p} and \ref{w-H2-est}, it suffices to prove the uniqueness. We employ the method of Yudovich.

 \vskip .1in
 Assume $(u,w,\pi)$ and $(\widetilde{u},\widetilde{w},\widetilde{\pi})$  are two solutions of \eqref{eq1}-\eqref{eq20} with the regularity specified in (\ref{regclass}). Consider their difference
 $$U=u-\widetilde{u},\,\,W=w-\widetilde{w},\,\,\Pi=\pi-\widetilde{\pi},$$ which solves the following initial- and boundary-value problem
\begin{equation}\label{uni}
\left\{\begin{array}{ll}
U_t+u\cdot\nabla U+U\cdot\nabla \widetilde{u}+\nabla\Pi=2\kappa{\nabla^{\perp}}W,\\
W_t+u\cdot\nabla W+U\cdot\nabla \widetilde{w}-\gamma\Delta W+4\kappa W=2\kappa \nabla\times U,\\
\nabla\cdot U=0,\\
U\cdot{\bm n}|_{\p\Omega}=W|_{\p\Omega}=0,\\
(U,W)(x,0)=0.
\end{array}\right.
\end{equation}
Dotting the first two equations with $(U,W)$ yields
\ben\label{uni-est}
&&\f12\f{d}{dt}(\|U\|_{L^2(D)}^2+\|W\|_{L^2(D)}^2)+\gamma\|\nabla W\|_{L^2(D)}^2+4\kappa\|W\|_{L^2(D)}^2\notag\\
&=&-\int_{D}  U\cdot\nabla \widetilde{u}\cdot Udx-\int_{D}  U\cdot\nabla \widetilde{w}\cdot Wdx\\
&&+2\kappa\int_{D}\nabla^{\perp}W\cdot Udx+2\kappa\int_{D}\nabla\times U\cdot Wdx.\notag
\een
By the divergence theorem and the boundary condition $W|_{\p\Omega}=0$,
\beno
&&2\kappa\, \int_{D}\nabla^{\perp}W\cdot Udx+2\kappa\int_{D}\nabla\times U\cdot Wdx\\
&=&4\kappa \,\int_{D}\nabla^{\perp}W\cdot U\,dx \le \frac{\gamma}{2}\,\|\nabla W\|_{L^2(D)}^2 + C\, \|U\|_{L^2(D)}^2.
\eeno
Since $\nabla \widetilde{u}$ is not known to be bounded
in $L^\infty$ while the corresponding vorticity $\widetilde{\Omega}$ is, we apply the Yudovich
approach to bound the first term on the right of (\ref{uni-est}).
Since $\widetilde{\Omega}\in  L^\infty(0,T; L^\infty(D))$, we have, by Lemma \ref{curl},
$$
\|\nabla \widetilde{u}\|_{L^q(D)} \le C\, q\, \left(\|\widetilde{\Omega}\|_{L^q(D)} + \|\widetilde{u}\|_{L^q(D)}\right)
\le C\, q\,\left(\|\widetilde{\Omega}\|_{L^\infty(D)} + \|\widetilde{\Omega}\|_{L^2(D)} + \|\widetilde{u}\|_{L^2(D)}\right).
$$
Therefore,
$$
L\equiv \sup_{q\ge 2} \frac{\|\nabla \widetilde{u}\|_{L^q(D)}}q <\infty.
$$
In addition, by Lemma \ref{P1},
\beno
M\equiv \|U\|_{L^\infty(D)} &\le& C\, \left(\|\nabla U\|^{\frac23}_{L^4(D)}\,\|U\|_{L^2(D)}^{\frac13} + \|U\|_{L^2(D)}\right)\\
&\le& C\,\left(\|\nabla u\|_{L^4(D)} + \|\nabla \widetilde{u}\|_{L^4(D)} + \|u\|_{L^2(D)} + \|\widetilde{u}\|_{L^2(D)}\right) <\infty.
\eeno
Therefore, by H\"{o}lder's inequality, for any $2\le q<\infty$,
$$
\left|\int_{D}  U\cdot\nabla \widetilde{u}\cdot U\,dx\right| \le C\, q\,L\, M^{\frac2q}\, \left(\|U\|_{L^2(D)}^2 + \delta\right)^{1-\frac1q},
$$
where $\delta>0$ is inserted here to justify some of the
later steps. By optimizing the bound above in terms of $q$, we obtain
$$
\left|\int_{D}  U\cdot\nabla \widetilde{u}\cdot U\,dx\right| \le C\,e\,
L\, (\|U\|_{L^2(D)}^2 + \delta)\, \ln\frac{M^2}{\|U\|_{L^2(D)}^2 + \delta}.
$$
To bound the second term on the right of (\ref{uni-est}), we
integrate by parts and invoke the boundary condition $W|_{\p\Omega}=0$
to obtain
\begin{equation*}\begin{split}
-\int_{D}  U\cdot\nabla \widetilde{w}\cdot Wdx &= \int_D \widetilde{w}\,U\cdot \nabla W\,dx \\
&\leq\| \widetilde{w}\|_{L^{\infty}(D)}\, \|U\|_{L^2(D)}\,\|\nabla W\|_{L^2(D)}\\
&\leq \frac{\gamma}{2}\, \|\nabla W\|_{L^2(D)}^2 + C\,
\|\widetilde{w}\|^2_{H^2(D)}\, \|U\|^2_{L^2(D)}.
\end{split}
\end{equation*}
Inserting the estimates above in \eqref{uni-est} yields
\beno
\frac{d}{dt}\left((\|U\|_{L^2(D)}^2 + \delta)+\|W\|_{L^2(D)}^2\right)
&\leq& C\,(1+ \|\widetilde{w}\|^2_{H^2(D)})\,(\|U\|_{L^2(D)}^2 + \delta)
\\
&& +  C\,
L\, (\|U\|_{L^2}^2 + \delta)\, \ln\frac{M^2}{\|U\|_{L^2(D)}^2 + \delta}.
\eeno
By Osgood's inequality, we obtain
\beno
&& (\|U(t)\|_{L^2(D)}^2 + \delta)+\|W(t)\|_{L^2(D)}^2 \\
&\le& e^{C\, \int_0^t(1+ \|\widetilde{w}\|^2_{H^2(D)})\,d\tau}
\left(\|U_0\|_{L^2(D)}^2 + \delta)+\|W_0\|_{L^2(D)}^2\right)^{e^{-C\,L\,t}}\\
&& \times \, e^{\int_0^t CL\, e^{-C\,L\,(t-\tau)}\,\ln\left(M^2 \exp(C\, \int_0^\tau(1+ \|\widetilde{w}\|^2_{H^2(D)})\,ds\right)\,d\tau}
\eeno
for any $t\in (0,T)$. Letting $\delta\to 0$ and noting that $U_0=W_0=0$,
we obtain the desired uniqueness $U=W\equiv 0$. This finishes the proof of Theorem \ref{T1}.
\end{proof}

\vskip .3in
\section{Higher regularity and proof of Theorem \ref{T2}}
\label{higher}
\setcounter{equation}{0}

\vskip .1in
The section proves Theorem \ref{T2}, the higher regularity of
$(u,w)$. These higher regularity bounds are achieved through
several steps. The first step is to prove the global bound
$$
\nabla u \in L^\infty(0, T; L^\infty(D)) \quad\mbox{and}\quad
\nabla \Omega \in L^\infty(0, T; L^p(D))\quad\mbox{with}\,\,2\le p<\infty
$$
via the equation of the combined quantity $Z$. The second step
is to show the global bound for
$$
w\in L^\infty(0, T; W^{2,p}(D))\quad \mbox{with}\,\,p\in [2, \infty) \quad\mbox{and}\quad
w\in L^\infty(0, T; H^3(D)).
$$
Finally we prove $w\in L^\infty(0, T; H^4(D))$. To do so,
we estimate $\|w_{tt}\|_{L^2(D)}$ and $\Omega\in L^\infty(0, T; H^2(D))$, and invoke the regularization estimates for elliptic equations (Lemma
\ref{elliptic}).

\vskip .1in
We start with the first step.
\begin{proposition}\label{nabla u-est}
Assume that $(u_0, w_0)$ satisfies the conditions in Theorem \ref{T2}.
Let $(u,w)$ be the corresponding solution of \eqref{eq1}-\eqref{eq20} guaranteed by Proposition \ref{weak2}. Then,  for any $T>0$ and $0<t\le T$, and for any $2\leq p<\infty$,
\begin{equation*}
\|\nabla u\|_{L^\infty(0,T; L^\infty(D))}+\|\nabla \Omega\|_{L^\infty(0,T; L^p(D))}\leq C.
\end{equation*}
where the constant $C$ depends only on $D, T$ and the initial data.	
\end{proposition}

\begin{proof}
Taking the first-order partial $\partial_i$ of \eqref{Z-equ} yields
\begin{equation}\label{de-Z-equ}
\partial_t\partial_iZ+u\cdot\nabla \partial_iZ+\partial_iu\cdot\nabla Z-\frac{4\kappa^2}{\gamma}\partial_iZ+
\frac{8\kappa^2}{\gamma}(1+\frac\kappa\gamma)\partial_iw=0.
\end{equation}
Multiplying \eqref{de-Z-equ} by $|\partial_iZ|^{p-2}\partial_iZ$, summing over $i$ and integrating on $D$, we have
\begin{equation*}\begin{split}
\frac1p\frac{d}{dt}\|\nabla Z\|_{L^p(D)}^p&\leq\|\nabla u\|_{L^\infty(D)}\|\nabla Z\|_{L^p(D)}^p+\frac{4\kappa^2}{\gamma}\|\nabla Z\|_{L^p(D)}^p\\
&\quad+\frac{8\kappa^2}{\gamma}(1+\frac\kappa\gamma)
\|\nabla w\|_{L^p(D)}\|\nabla Z\|_{L^p(D)}^{p-1},
\end{split}
\end{equation*}
which also implies
\begin{equation}\label{de-Z-lp}
\frac{d}{dt}\|\nabla Z\|_{L^p(D)}\leq C\left(\frac{\kappa^2}{\gamma}+\|\nabla u\|_{L^\infty(D)}\right)\|\nabla Z\|_{L^p(D)}+\frac{8\kappa^2}{\gamma}\left(1+\frac\kappa\gamma\right)
\|\nabla w\|_{L^p(D)}.
\end{equation}
By Lemma \ref{log-inequ} and the Sobolev embedding $W^{1,p}(D)\hookrightarrow C^{\alpha}(D)$ for $p>2$,
\begin{equation*}
\begin{split}
\|\nabla u\|_{L^\infty(D)}\leq& C\, \|\Omega\|_{L^\infty(D)}\log(e+\|\Omega\|_{W^{1,p}(D)}),
\end{split}
\end{equation*}
which, together with the definition of $Z$, yields
\begin{equation}\label{log}\begin{split}
\|\nabla u\|_{L^\infty(D)}
&\leq C\|\Omega\|_{L^\infty(D)}\log(e+\|\nabla Z\|_{L^{p}(D)}+\|Z\|_{L^{p}(D)}+\frac{2\kappa}{\gamma}\|w\|_{W^{1,p}(D)}).
\end{split}
\end{equation}
Inserting (\ref{log}) in (\ref{de-Z-lp}) yields
\begin{equation*}
\frac{d}{dt}\|\nabla Z\|_{L^p(D)}\leq C(1+\log(e+\|\nabla Z\|_{L^{p}(D)})) \,\|\nabla Z\|_{L^p(D)} + C,
\end{equation*}
where we have invoked the regularity bounds for $(u, w)$ from the
previous section. We obtain, via Gronwall's inequality, the global bound
\begin{equation}\label{NABLAZ-Lp}
\|\nabla Z\|_{L^\infty(0,T;L^p(D))}\leq C.
\end{equation}
By the definition of $Z$, for any $2<p<\infty$,
\ben\label{nablaomega}
\|\nabla \Omega\|_{L^\infty(0,T;L^p(D))}\leq \|\nabla Z\|_{L^\infty(0,T;L^p(D))}+\frac{2\kappa}{\gamma}\|\nabla w\|_{L^\infty(0,T;L^p(D))} \le C,
\een
which, together with \eqref{log}, implies
\begin{equation*}
\|\nabla u\|_{L^\infty(0,T;L^\infty(D))}\leq C.
\end{equation*}
By the way, the bound of $\|\nabla \Omega\|_{L^\infty(0,T;L^2(D))}$ can be inferred by \eqref{nablaomega}, H\"{o}lder's inequality and the boundedness of domain $D$ directly. This completes the proof of Proposition \ref{nabla u-est}.
\end{proof}

\vskip .1in
Our next goal is to show the global bound for $\|w\|_{W^{2,p}(D)}$ and
$\|w\|_{H^3(D)}$. To avoid the boundary effects, we make use of the
estimates of time derivatives and the regularization bounds in Lemma \ref{elliptic}.

\begin{proposition}\label{w-H3-est}
Assume that $(u_0, w_0)$ satisfies the conditions in Theorem \ref{T2}.
Let $(u,w)$ be the corresponding solution of \eqref{eq1}-\eqref{eq20} guaranteed by Proposition \ref{weak2}. Then,  for any $T>0$ and $0<t\le T$, and for any $2\leq p<\infty$,
\begin{equation}
\|w\|_{L^\infty(0,T; \,W^{2,p}(D))}+\|w\|_{L^\infty(0,T;\, H^3(D))}\leq C,
\end{equation}
where the constant $C$ depends only on $D, T, p$ and the initial data.
\end{proposition}

\begin{proof} To prove this proposition, we first show that
$$
\|\nabla w_t\|_{L^\infty(0,T; \, L^2(D))}\leq C.
$$
By Lemma \ref{heat}, for any $2\leq p<\infty$,
\begin{equation*}\begin{split}
&\int_{0}^{T}\|\Delta w\|_{L^{p}(D)}^pdt\\
\leq& C\int_{0}^{T}\big[\|w_0\|_{H^{2}(D)}^p+\| w\|_{L^{p}(D)}^p+\|\Omega\|_{L^{p}(D)}^p
+\|u\cdot\nabla w\|_{L^{p}(D)}^p\big]dt\\
\leq& C\int_{0}^{T}\big[\|w_0\|_{H^{2}(D)}^p+\|w\|_{H^{1}(D)}^p+\|\Omega\|_{L^{p}(D)}^p
+\|u\|_{L^{2p}(D)}^p\|\nabla w\|_{L^{2p}(D)}^p\big]dt.
\end{split}
\end{equation*}	
Due to the embedding
$$
\|\nabla w\|_{L^{2p}(D)} \le C\, \|w\|_{H^{2}(D)}
$$
and the global bound on $\|w\|_{L^\infty(0,T; \, H^2(D))}$ in
Proposition \ref{w-H2-est}, we obtain
\begin{equation}\label{w-w-2,p}\begin{split}
\int_{0}^{T}\|w\|_{W^{2,p}(D)}^p\,dt\leq C.
\end{split}
\end{equation}		
Multiplying \eqref{wt-equ} by $-\Delta w_t$ and integrating on $D$, it follows that
\begin{equation}\label{wt-H1}\begin{split}
&\frac12\frac{d}{dt}\|\nabla w_t\|_{L^2(D)}^2+\gamma\|\Delta w_t\|_{L^2(D)}^2+4\kappa\|\nabla w_t\|_{L^2(D)}^2\\=&\int_{D}  u\cdot\nabla w_t \Delta w_tdx-\int_{D} u_t\cdot\nabla w \Delta w_tdx-2\kappa\int_{D} \Omega_t\Delta w_t dx.
\end{split}
\end{equation}
By H\"{o}lder's inequality, Sobolev's embedding and Lemma \ref{curl},
\begin{equation*}\begin{split}
\int_{D}  u\cdot\nabla w_t \Delta w_tdx&\leq \|u\|_{L^\infty(D)}\|\nabla w_t\|_{L^2(D)}\|\Delta w_t\|_{L^2(D)}\\&\leq \frac\gamma6\|\Delta w_t\|_{L^2(D)}^2+\frac C\gamma\|u\|_{H^2(D)}^2\|\nabla w_t\|_{L^2(D)}^2\\
&\leq \frac\gamma6\|\Delta w_t\|_{L^2(D)}^2+\frac C\gamma[\|u\|_{L^2(D)}^2+\|\nabla \Omega\|_{L^2(D)}^2]\|\nabla w_t\|_{L^2(D)}^2
\end{split}
\end{equation*}
and
\begin{equation*}\begin{split}
-\int_{D}  u_t\cdot\nabla w \Delta w_tdx&\leq \|u_t\|_{L^2(D)}\|\nabla w\|_{L^\infty(D)}\|\Delta w_t\|_{L^2(D)}\\&\leq \frac\gamma6\|\Delta w_t\|_{L^2(D)}^2+\frac C\gamma\|u_t\|_{L^2(D)}^2
\|w\|_{W^{2,p}(D)}^2.
\end{split}
\end{equation*}
By the vorticity equation \eqref{eq31} and Lemma \ref{curl},
\begin{equation*}\begin{split}
-2\kappa\int_{D}  \Omega_t\Delta w_t dx&=2\kappa\int_{D} u\cdot\nabla\Omega\Delta w_t dx+4\kappa^2\int_{D} \Delta w\Delta w_t dx\\&\leq\frac\gamma6\|\Delta w_t\|_{L^2(D)}^2+ \frac{C\kappa}{\gamma}\|u\|_{H^2(D)}^2\|\nabla \Omega\|_{L^2(D)}^2+\frac{C\kappa^2}{\gamma}\|\Delta w\|_{L^2(D)}^2\\
&\leq\frac\gamma6\|\Delta w_t\|_{L^2(D)}^2+ \frac{C\kappa}{\gamma}[\|u\|_{L^2(D)}^2+\|\nabla \Omega\|_{L^2(D)}^2]\|\nabla \Omega\|_{L^2(D)}^2\\
&\quad+\frac{C\kappa^2}{\gamma}\|w\|_{H^2(D)}^2.
\end{split}
\end{equation*}
Inserting the estimates above in \eqref{wt-H1}
and invoking the regularity bounds obtained before, we have
\begin{equation} \label{ntu}
\|\nabla w_t\|_{L^2(D)}\leq C.
\end{equation}	
By Lemma \ref{elliptic},
\begin{equation*}
\begin{split}
\|w\|_{W^{2,p(D)}}&\leq C\big[\|w_t\|_{L^p(D)}+\|u\cdot\nabla w\|_{L^p(D)}+\|w\|_{L^p(D)}+\|\Omega\|_{L^p(D)}\big]\\
&\leq C\big[\|w_t\|_{H^1(D)}+\|u\|_{L^{2p}(D)}\|\nabla w\|_{L^{2p}(D)}+\|w\|_{H^1(D)}+\|\Omega\|_{L^p(D)}\big]\\
&\leq C\big[\|w_t\|_{H^1(D)}+\|u\|_{H^1(D)}\|\nabla w\|_{H^1(D)}+ \|w\|_{H^1(D)}+\|\Omega\|_{L^p(D)}\big].
\end{split}
\end{equation*}
According to (\ref{ntu}) and Proposition \ref{w-w2p},
\begin{equation}\label{w-W2p-est}
\|w\|_{L^\infty(0, T; \,W^{2,p}(D))}\leq C.
\end{equation}	
As a consequence, by Lemma \ref{elliptic},
\begin{equation}\label{w-H3}\begin{split}
\|w\|_{H^3(D)}&\leq C\big[\|w_t\|_{H^1(D)}+\|u\cdot\nabla w\|_{H^1(D)}+\|w\|_{H^1(D)}+\|\Omega\|_{H^1(D)}\big]\\
&\leq C\big[\|w_t\|_{H^1(D)}+\|u\|_{L^\infty(D)}\|\nabla w\|_{H^1(D)}+\|\nabla w\|_{L^\infty(D)}\|u\|_{H^1(D)}\\
&\quad+\|w\|_{H^1(D)}+\|\Omega\|_{H^1(D)}\big]\\
&\leq C\big[\|w_t\|_{H^1(D)}+\|u\|_{H^2(D)}\|w\|_{H^2(D)}+\| w\|_{W^{2,p}(D)}\|u\|_{H^1(D)}\\
&\quad+\|w\|_{H^1(D)}+\|\Omega\|_{H^1(D)}\big]\\
&\leq C.
\end{split}
\end{equation}
This completes the proof of Proposition \ref{w-H3-est}.
\end{proof}

\vskip .1in
Finally we prove the global $H^4$-bound for $w$ by making full use of the structure of the micropolar equations and classical elliptic regularization theory.

\begin{proposition}\label{w-H4-est}
Assume that $(u_0, w_0)$ satisfies the conditions in Theorem \ref{T2}.
Let $(u,w)$ be the corresponding solution of \eqref{eq1}-\eqref{eq20} guaranteed by Proposition \ref{weak2}. Then,  for any $T>0$ and
$0<t\le T$,
\begin{equation*}
\|w\|_{L^\infty(0,T; H^4(D))}\leq C,
\end{equation*}
where the constant $C$ depends only on $D, T$ and the initial data.	
\end{proposition}

\begin{proof}
By Lemma \ref{elliptic},
\begin{equation*}\begin{split}
\|w\|_{H^4(D)}&\leq C\big[\|w_t\|_{H^2(D)}+\|u\cdot\nabla w\|_{H^2(D)}+\|w\|_{H^2(D)} +\, \|\Omega\|_{H^2(D)}\big].
\end{split}
\end{equation*}
$\|w\|_{H^2(D)}$ is globally bounded due to Proposition \ref{w-H2-est}.
To bound $\|u\cdot\nabla w\|_{H^2(D)}$, we apply Lemma \ref{curl}, Lemma \ref{commutator-est} and the Sobolev embedding $H^2(D)\hookrightarrow L^\infty(D)$ to get
\begin{equation*}\begin{split}
\|u\cdot\nabla w\|_{H^2(D)}&\leq C\big[\|u\|_{L^\infty(D)}\|\nabla w\|_{H^2(D)}+\|\nabla w\|_{L^\infty(D)}\|u\|_{H^2(D)}\big]\\
&\leq C\|u\|_{H^2(D)}\|\nabla w\|_{H^2(D)}\\
&\leq C\big[\|\Omega\|_{H^1(D)}+\|u\|_{L^2(D)}\big]\|w\|_{H^3(D)} \le C.
\end{split}
\end{equation*}
according to Propositions \ref{nabla u-est}-\ref{w-H3-est}. Our main
efforts are devoted to bounding
$$
\|w_t\|_{H^2(D)} \quad\mbox{and}\quad \|\Omega\|_{H^2(D)}.
$$
The following two lemmas establish these desired global bounds. With
the help of these lemmas, we obtain
$$
\|w\|_{L^\infty(0,T;\, H^4(D))} \le C.
$$
The two lemmas and their proofs are given below. This completes the
proof of Proposition \ref{w-H4-est}.
\end{proof}

\vskip .1in
We have used the fact that $\|w_t\|_{H^2(D)}$ is globally bounded.
The following lemma states this fact and then prove this fact.

\begin{lemma}\label{Vt-L2}
Assume that $(u_0, w_0)$ satisfies the conditions in Theorem \ref{T2}.
Let $(u,w)$ be the corresponding solution of \eqref{eq1}-\eqref{eq20} guaranteed by Proposition \ref{weak2}. Then,  for any $T>0$ and
$0<t\le T$,
\begin{equation*}
\|\Omega_{t}\|_{L^\infty(0,T; L^p(D))}\leq C
\quad\mbox{and}\quad
\|w_t\|_{L^\infty(0, T; \,H^2(D)} \leq C,
\end{equation*}
where the constant $C$ depends only on $D, T$ and the initial data.	
\end{lemma}

\begin{proof}
Applying Lemma \ref{elliptic} to \eqref{wt-equ} yields
\begin{equation}\label{w-H4}\begin{split}
\|w_{t}\|_{H^2(D)}&\leq C\big[\|w_{tt}\|_{L^2(D)}+\|u\cdot\nabla w_t\|_{L^2(D)}+\|u_t\cdot\nabla w\|_{L^2(D)}\\
&\quad+\|w_t\|_{L^2(D)}+\|\Omega_t\|_{L^2(D)}\big]\\
&\leq C\big[\|w_{tt}\|_{L^2(D)}+\|u\|_{H^2(D)}\|\nabla w_t\|_{L^2(D)}+\|u_t\|_{L^2(D)}\|\nabla w\|_{H^2(D)}\\
&\quad+\|w_t\|_{L^2(D)}+\|\Omega_t\|_{L^2(D)}\big].
\end{split}
\end{equation}
Clearly, the terms $\|u\|_{H^2(D)}\|\nabla w_t\|_{L^2(D)}+\|u_t\|_{L^2(D)}\|\nabla w\|_{H^2(D)}+\|w_t\|_{L^2(D)}$ are all bounded. Therefore,  it suffices to bound $\|\Omega_t\|_{L^2(D)}$ and $\|w_{tt}\|_{L^2(D)}$.

\vskip .1in
Multiplying \eqref{eq31} by $|\Omega_{t}|^{p-2}\Omega_{t}$ and integrating on $D$, we have
\begin{equation}\label{Omegat-est}
\begin{split}
\|\Omega_{t}\|_{L^p(D)}^p&=-\int_{D}  u\cdot\nabla \Omega |\Omega_{t}|^{p-2}\Omega_{t} dx - 2\kappa \int_{D}  \Delta w|\Omega_{t}|^{p-2}\Omega_{t}dx.
\end{split}
\end{equation}
By H\"{o}lder's inequality and the embeddings $H^2(D)\hookrightarrow L^\infty(D)$ and $H^1(D)\hookrightarrow L^p(D)$,
\begin{equation*}\begin{split}
-\int_{D}  u\cdot\nabla \Omega |\Omega_{t}|^{p-2}\Omega_{t} dx&\leq
\|u\|_{L^\infty(D)}\|\nabla \Omega\|_{L^p(D)}\|\Omega_{t}\|_{L^p(D)}^{p-1}\\
&\leq\|u\|_{H^2(D)}\|\nabla \Omega\|_{L^p(D)}\|\Omega_{t}\|_{L^p(D)}^{p-1}
\end{split}
\end{equation*}
and
\begin{equation*}\begin{split}
2\kappa \int_{D}  \Delta w|\Omega_{t}|^{p-2}\Omega_{t}dx&\leq
C\kappa\|\Delta w\|_{L^p(D)}\|\Omega_{t}\|_{L^p(D)}^{p-1}\\
&\leq C\kappa\|w\|_{H^3(D)}\|\Omega_{t}\|_{L^p(D)}^{p-1}.
\end{split}
\end{equation*}
Inserting the estimates above in \eqref{Omegat-est} yields
\begin{equation*}\begin{split}
\|\Omega_{t}\|_{L^p(D)}&\leq C\big[\|u\|_{H^2(D)}\|\nabla \Omega\|_{L^p(D)}+\| w\|_{H^3(D)}\big]\\
&\leq C\big[(\|\Omega\|_{H^1(D)}+\|u\|_{L^2(D)})\|\nabla \Omega\|_{L^p(D)}+\| w\|_{H^3(D)}\big] \le C.
\end{split}\end{equation*}

\vskip .1in
Next we prove the global bound
\begin{equation*}
\|u_{tt}\|_{L^\infty(0,T; L^2(D))}+\|w_{tt}\|_{L^\infty(0,T; L^2(D))}
\leq C.
\end{equation*}
Taking the temporal derivative of the velocity equation in  $(\ref{eq1})$ yields
\begin{equation}\label{utt-equ}
u_{tt}+u\cdot\nabla u_t+u_t\cdot\nabla u+\nabla p_t=-2\kappa\nabla^{\bot}w_t.
\end{equation}
Dotting \eqref{utt-equ} with $u_{tt}$ yields
\begin{equation*}\begin{split}
\|u_{tt}\|_{L^2(D)}^2&=-\int_{D}  u\cdot\nabla u_t\cdot u_{tt} dx-\int_{D} u_t\cdot\nabla u \cdot u_{tt} dx-2\kappa \int_{D} u_{tt}\cdot\nabla^{\bot}w_t dx\\
&\leq\frac12\|u_{tt}\|_{L^2(D)}^2+C\big[\|u\cdot\nabla u_t\|_{L^2(D)}^2+\|u_t\cdot\nabla u\|_{L^2(D)}^2+\|\nabla^{\bot}w_t\|_{L^2(D)}^2\big]\\
&\leq\frac12\|u_{tt}\|_{L^2(D)}^2+C\big[\|u\|_{L^\infty(D)}^2\|\nabla u_t\|_{L^2(D)}^2+\|u_t\|_{L^2(D)}^2\|\nabla u\|_{L^\infty(D)}^2\\
&\quad+\|\nabla w_t\|_{L^2(D)}^2\big]\\
&\leq\frac12\|u_{tt}\|_{L^2(D)}^2+C\big[\|u\|_{H^2(D)}^2(\|u_t\|_{L^2(D)}^2+\|\Omega_t\|_{L^2(D)}^2)\\
&\quad+\|u_t\|_{L^2(D)}^2\|u\|_{H^2(D)}^2+\|\nabla w_t\|_{L^2(D)}^2\big]\\
&\leq\frac12\|u_{tt}\|_{L^2(D)}^2+C\big[(\|u\|_{L^2(D)}^2+\|\nabla \Omega\|_{L^2(D)}^2)(\|u_t\|_{L^2(D)}^2+\|\Omega_t\|_{L^2(D)}^2)\\
&\quad+\|\nabla w_t\|_{L^2(D)}^2\big],
\end{split}
\end{equation*}
which implies
\begin{equation}\label{utt-est}
\|u_{tt}\|_{L^2(D)}\leq C.
\end{equation}	
To estimate $\|w_{tt}\|_{L^2(D)}$, we take the second-order temporal derivative of \eqref{wt-equ}
\begin{equation}\label{wttt-equ}
w_{ttt}+u\cdot\nabla w_{tt}+2u_t\cdot\nabla w_t+u_{tt}\cdot\nabla w+4\kappa w_{tt}=\gamma \Delta w_{tt}+2\kappa\Omega_{tt}.
\end{equation}
Multiplying \eqref{wttt-equ} by $w_{tt}$ and integrating on $D$, we have
\begin{equation}\label{wtt-est}
\begin{split}
&\frac12\frac{d}{dt}\|w_{tt}\|_{L^2(D)}^2+\gamma\|\nabla w_{tt}\|_{L^2(D)}^2+4\kappa\|w_{tt}\|_{L^2(D)}^2\\=&-2\int_{D}  u_t\cdot\nabla w_t w_{tt}dx-\int_{D}  u_{tt}\cdot\nabla w w_{tt}dx+2\kappa\int_{D}  \Omega_{tt}w_{tt} dx.
\end{split}
\end{equation}
By integration by parts, H\"{o}lder's inequality, Young's inequality, Lemma \ref{curl} and the embedding $H^1(D)\hookrightarrow L^4(D)$,
\begin{equation}\label{1-est}
\begin{split}
-2\int_{D}  u_t\cdot\nabla w_t w_{tt}dx&=2\int_{D}  u_t\cdot\nabla w_{tt}w_tdx\\&\leq
\|u_{t}\|_{L^4(D)}\|w_{t}\|_{L^4(D)}\|\nabla w_{tt}\|_{L^2(D)}\\&\leq
\frac\gamma6\|\nabla w_{tt}\|_{L^2(D)}^2+\frac C\gamma\|u_{t}\|_{H^1(D)}^2\|w_{t}\|_{H^1(D)}^2\\&\leq
\frac\gamma6\|\nabla w_{tt}\|_{L^2(D)}^2+\frac C\gamma[\|\Omega_{t}\|_{L^2(D)}^2+\|u_{t}\|_{L^2(D)}^2]\|w_{t}\|_{H^1(D)}^2,
\end{split}
\end{equation}
\begin{equation}\label{2-est}
\begin{split}
-\int_{D}  u_{tt}\cdot\nabla w w_{tt}dx&=\int_{D}  u_{tt}\cdot\nabla w_{tt} wdx\\&\leq
\|u_{tt}\|_{L^2(D)}\|w\|_{L^\infty(D)}\|\nabla w_{tt}\|_{L^2(D)}\\&\leq
\frac\gamma6\|\nabla w_{tt}\|_{L^2(D)}^2+\frac C\gamma\|u_{tt}\|_{L^2(D)}^2\|w\|_{H^2(D)}^2
\end{split}
\end{equation}
and
\begin{equation}\label{3-est}
\begin{split}
2\kappa\int_{D}  \Omega_{tt}w_{tt} dx&=-2\kappa\int_{D}  u_{tt}\cdot\nabla^{\bot}w_{tt} dx\\
&\leq\frac\gamma6\|\nabla w_{tt}\|_{L^2(D)}^2+\frac {C\kappa^2}{\gamma}\|u_{tt}\|_{L^2(D)}^2.
\end{split}
\end{equation}
Inserting \eqref{1-est}-\eqref{3-est} in \eqref{wtt-est} yields
\begin{equation}\label{wtt-est-1}
\begin{split}
\frac{d}{dt}\|w_{tt}\|_{L^2(D)}^2&\leq C(\|\Omega_{t}\|_{L^2(D)}^2+\|u_{t}\|_{L^2(D)}^2)\|w_{t}\|_{H^1(D)}^2\\
&\quad+C(1+\|u_{tt}\|_{L^2(D)}^2)\|w\|_{H^2(D)}^2.
\end{split}
\end{equation}
Integrating in time yields the desired estimates. This proves
Lemma \ref{Vt-L2}.
\end{proof}

\vskip .1in
We now move on to the second lemma asserting the global bound for
$\|\Omega\|_{H^2(D)}$.

\begin{lemma}\label{V-H2}
Assume that $(u_0, w_0)$ satisfies the conditions in Theorem \ref{T2}.
Let $(u,w)$ be the corresponding solution of \eqref{eq1}-\eqref{eq20} guaranteed by Proposition \ref{weak2}. Then,  for any $T>0$ and
$0<t\le T$,
\begin{equation*}
\|\Omega\|_{L^\infty(0,T; H^2(D))}\leq C.
\end{equation*}
where the constant $C$ depends only on $D, T$ and the initial data.
\end{lemma}

\begin{proof} Taking $\partial_{i}^2$ of \eqref{Z-equ} yields
\begin{equation}\label{2-Z-equ}
\partial_t\partial_{i}^2Z+u\cdot\nabla \partial_{i}^2Z-\frac{4\kappa^2}{\gamma}\partial_{i}^2Z
+\frac{8\kappa^2}{\gamma}(1+\frac\kappa\gamma)\partial_{i}^2w=u\cdot\nabla \partial_{i}^2Z-\partial_{i}^2(u\cdot\nabla Z).
\end{equation}
Multiplying  \eqref{2-Z-equ} by $\partial_{i}^2Z$ and integrating on $D$, we have
\begin{equation}\label{2-Z-est}
\begin{split}
\frac12\frac{d}{dt}\|\partial_i^2 Z\|_{L^2(D)}^2&
=\frac{4\kappa^2}{\gamma}\|\partial_i^2 Z\|_{L^2(D)}^2-
\frac{8\kappa^2}{\gamma}(1+\frac\kappa\gamma)
\int_{D}  \partial_{i}^2w \partial_{i}^2Zdx\\
&\quad+\int_{D} \big(u\cdot\nabla \partial_{i}^2Z-\partial_{i}^2(u\cdot\nabla Z)\big)\partial_i^2 Z dx.
\end{split}
\end{equation}
By $\nabla\cdot u=0$ and the commutator estimate \eqref{c-1},
\begin{equation}\label{c-3}
\begin{split}
&\int_{D} \big(u\cdot\nabla \partial_{i}^2Z-\partial_{i}^2(u\cdot\nabla Z)\big)\partial_i^2 Z dx\\
=&\int_{D} \big(u\cdot\partial_{i}^2 \nabla Z-\partial_{i}^2 \nabla\cdot(u Z)\big)\partial_i^2 Z dx\\
\leq&C\big(\|\nabla u\|_{L^{\infty}(D)}\|Z\|_{H^2(D)}\|\partial_i^2 Z\|_{L^2(D)}+\|u\|_{H^3(D)}
\|Z\|_{L^\infty(D)}\|\partial_i^2 Z\|_{L^2(D)}\big).
\end{split}
\end{equation}
Then, by Lemma \ref{curl} and definition of $Z$, one has
\begin{equation*}\begin{split}
\|u\|_{H^3(D)}\leq C\big(\|\Omega\|_{H^2(D)}+\|u\|_{L^2(D)}\big)\leq C\big(\|Z\|_{H^2(D)}+\|w\|_{H^2(D)}+\|u\|_{L^2(D)}\big),
\end{split}
\end{equation*}	
which together with \eqref{Z--infty} and \eqref{c-3} yields
\begin{equation*}\begin{split}
&\int_{D} \big(u\cdot\nabla \partial_{i}^2Z-\partial_{i}^2(u\cdot\nabla Z)\big)\partial_i^2 Z dx\\
\leq&\|\nabla u\|_{L^{\infty}(D)}\|Z\|_{H^2(D)}\|\partial_i^2 Z\|_{L^2(D)}+
\|Z\|_{H^2(D)}\|\partial_i^2 Z\|_{L^2(D)}\\
&+(\|w\|_{H^2(D)}+\|u\|_{L^2(D)})\|\partial_i^2 Z\|_{L^2(D)}.
\end{split}
\end{equation*}
By H\"{o}lder's inequality,
\begin{equation*}
-\frac{8\kappa^2}{\gamma}(1+\frac\kappa\gamma)
\int_{D} \partial_{i}^2w \partial_{i}^2Zdx\leq\frac{8\kappa^2}{\gamma}(1+\frac\kappa\gamma)
\|\partial_{i}^2 w\|_{L^2(D)}\|\partial_{i}^2 Z\|_{L^2(D)}.
\end{equation*}
Inserting the estimates in \eqref{2-Z-est} and summing up by $i$, we have
\begin{equation*}
\begin{split}
\frac12\frac{d}{dt}\|\nabla^2 Z\|_{L^2(D)}^2\leq&
C(1+\|\nabla u\|_{L^{\infty}(D)}+\|w\|_{H^2(D)}^2+\|u\|_{L^2(D)}^2)\\
&\times(\|\nabla^2Z\|_{L^2(D)}^2+\|\nabla Z\|_{L^2(D)}^2+\|Z\|_{L^2(D)}^2).
\end{split}
\end{equation*}
Gronwall's inequality implies
\begin{equation*}
\|\nabla^2 Z\|_{L^2(D)}\leq C.
\end{equation*}
By the definition of $Z$ and Proposition \ref{w-H2-est}, we have
\begin{equation*}
\|\Omega\|_{H^2(D)}\leq \|Z\|_{H^2}+\|w\|_{H^2(D)}\leq C.
\end{equation*}
This completes the proof of Lemma \ref{V-H2}.
\end{proof}

\vskip .3in
\section{Large time behavior }
\label{Large}
\setcounter{equation}{0}

This section proves Theorem \ref{T3}, namely the large-time estimate stated in \eqref{eq7} under the condition that $\gamma >4 \kappa$.

\begin{proof}[Proof of Theorem \ref{T3}]
To begin with,	we take the inner product of $(\ref{eq5})$ with $(u, w)$
to obtain
\begin{equation}\label{decay-l2}
\begin{split}
&\frac12\frac{d}{dt}(\|u\|^{2}_{L^2(D)}+\|w\|^{2}_{L^2(D)})+\kappa \|u\|^{2}_{L^2(D)}+\gamma\|\nabla w\|_{L^2(D)}^2\\
=&2\kappa \int_D u\cdot\nabla^{\bot}w\,dx
+ 2\kappa\int_D w\nabla\times u\,dx
= 4 \kappa\,\int_D u\cdot\nabla^{\bot}w\,dx
\end{split}
\end{equation}
where we have used (\ref{bbb}) to combine the two
terms on the right. By H\"{o}lder's inequality,
$$
4 \kappa \,\int_D u\cdot\nabla^{\bot}w\,dx \le 4\kappa \|u\|_{L^2(D)}\|\nabla w \|_{L^2(D)}
 \leq  \frac{8 \kappa^2}{\gamma + 4 \kappa} \|u(t)\|_{L^2(D)}^2+ \frac{\gamma + 4 \kappa}{2} \|\nabla w(t)\|^2_{L^2(D)}.
$$
Therefore,
\begin{equation}\label{decaymm}
\frac{d}{dt}
\left( \|u(t)\|^2_{L^2(D)}+\|w(t)\|_{L^2(D)}^2\right) + \frac{2\kappa(\gamma-4\kappa)}{\gamma+ 4\kappa} \|u(t)\|^2_{L^2(D)}
  + (\gamma-4\kappa)\,\|\nabla  w(t)\|_{L^2(D)}^2 \le 0.
\end{equation}
Noticing that $w|_{\partial D}=0$, we have, by Lemma \ref{poincare},
\begin{equation*}
\|w\|_{L^2(D)}^2\leq C\|\nabla w\|_{L^2(D)}^2,
\end{equation*}
which, together with \eqref{decaymm}, yields
\begin{equation*}
\frac12\frac{d}{dt}(\|u\|^{2}_{L^2(D)}+\|w\|^{2}_{L^2(D)})+ \frac{2\kappa(\gamma-4\kappa)}{\gamma+ 4\kappa} \|u\|^{2}_{L^2(D)}+ \frac{(\gamma-4\kappa)}{C}\| w\|_{L^2(D)}^2 \leq 0.
\end{equation*}
Since $\gamma> 4\kappa$, it is then clear that
\begin{equation}\label{decay-l2-est}
\|u\|^{2}_{L^2(D)}+\|w\|^{2}_{L^2(D)}\leq e^{-C_0 t}(\|u_0\|^{2}_{L^2(D)}+\|w_0\|^{2}_{L^2(D)}),
\end{equation}
where $C_0$ is given by
$$
C_0 =\min\left\{\frac{2\kappa(\gamma-4\kappa)}{\gamma+ 4\kappa}, \, \frac{(\gamma-4\kappa)}{C}\right\}.
$$
In addition, multiplying (\ref{decaymm})
by $e^{\frac{C_0}{2}t}$ and integrating in time yields
\ben
&&\frac{2\kappa(\gamma-4\kappa)}{\gamma+ 4\kappa} \int_0^t e^{\frac{C_0}{2}\tau}\,\|u(\tau)\|^{2}_{L^2(D)}\,d\tau  + (\gamma-4\kappa)\,\int_0^te^{\frac{C_0}{2}\tau}\, \|\nabla w(\tau)\|_{L^2(D)}^2\,d\tau \notag \\
&&\qquad\qquad \le \|u_0\|^{2}_{L^2(D)}+\|w_0\|^{2}_{L^2(D)}. \label{inttt}
\een

\vskip .1in
We now turn to the decay of the gradient of $(u, w)$.  Multiplying (\ref{eq31}) by $\Omega$ and the $w$-equation in $(\ref{eq5})$ by $-\Delta w$ and then integrating on $D$, one has
\begin{equation}\label{decay-H1}
\begin{split}
&\frac12\frac{d}{dt}(\|\Omega\|^{2}_{L^2(D)}+\|\nabla w\|^{2}_{L^2(D)})+\kappa \|\Omega\|^{2}_{L^2(D)}+\gamma\|\Delta w\|_{L^2(D)}^2\\
=&-4\kappa\int_{D} \Delta w \Omega\, dx + \int_{D} u\cdot\nabla w
\Delta w\, dx.
\end{split}
\end{equation}
By Young's inequality,
$$
4\kappa\, \left|\int_{D}  \Delta w\, \Omega dx\right|
\le \frac{16 \kappa^2}{\gamma + 12 \kappa}\|\Omega\|_{L^2(D)}^2 + \frac{\gamma + 12 \kappa}{4}\|\Delta w\|^2_{L^2(D)}.
$$
By H\"{o}lder's inequality, Corollary \ref{C1},  Lemma \ref{curl} and Young's inequality,
\begin{equation*} \begin{split}
\int_{D} u\cdot\nabla w\Delta w dx\leq& C\|u\|_{L^4(D)}\|\nabla w\|_{L^4(D)}\|\Delta w\|_{L^2(D)}\\ \leq&
C\|u\|_{L^2(D)}^{\frac12}\|\nabla u\|_{L^2(D)}^{\frac12}(\|\nabla w\|_{L^2(D)}^{\frac12}\|\nabla^2 w\|_{L^2(D)}^{\frac12}\\
&+\|\nabla w\|_{L^2(D)})\|\Delta w\|_{L^2(D)}\\
\leq&C\|u\|_{L^2(D)}^{\frac12}(\|\Omega\|_{L^2(D)}^{\frac12}+ \|u\|_{L^2(D)}^{\frac12})(\|\nabla w\|_{L^2(D)}^{\frac12}\|\Delta w\|_{L^2(D)}^{\frac12}\\
&+\|\nabla w\|_{L^2(D)})\|\Delta w\|_{L^2(D)}\\
\leq&   \frac{\gamma-4\kappa}{4}\, \|\Delta w\|_{L^2(D)}^2\\
&+C(1+\|u\|_{L^2(D)}^2)(\|u\|_{L^2(D)}^2+\|\Omega\|_{L^2(D)}^2)\|\nabla w\|_{L^2(D)}^2.
\end{split}
\end{equation*}
Inserting the estimates above in \eqref{decay-H1}, we have
\begin{equation}\label{H1-decay}\begin{split}
&\frac{d}{dt}(\|\Omega\|^{2}_{L^2(D)}+\|\nabla w\|^{2}_{L^2(D)})
+ \frac{4 \kappa(\gamma-4\kappa)}{\gamma +12 \kappa}\,\|\Omega\|^{2}_{L^2(D)}+2(\gamma-4\kappa)\,\|\Delta w\|_{L^2(D)}^2\\
\leq& C(1+\|u\|_{L^2(D)}^2)(\|u\|_{L^2(D)}^2+\|\Omega\|_{L^2(D)}^2)\|\nabla w\|_{L^2(D)}^2.
\end{split}
\end{equation}
We note that $\gamma>4\kappa$. Applying Gronwall's inequality and using
(\ref{inttt}), we have
\ben
&&\|\Omega\|^{2}_{L^2(D)}+\|\nabla w\|^{2}_{L^2(D)}
\le  e^{C(1+\|u\|^2_{L^2(D)})\,\int_0^t\|\nabla w(\tau)\|_{L^2(D)}^2\,d\tau} \left(\|\Omega_0\|^{2}_{L^2(D)}+\|\nabla w_0\|^{2}_{L^2(D)} \right)
\notag\\ &&\,\,\, +\,e^{C(1+\|u\|^2_{L^2(D)})\,\int_0^t\|\nabla w(\tau)\|_{L^2(D)}^2\,d\tau}\left((\|u\|^2_{L^2(D)}+\|u\|^2_{L^4(D)})
\int_0^t\|\nabla w(\tau)\|_{L^2(D)}^2d\tau \right) \notag\\
&&\leq C(u_0, w_0), \label{ggg}
\een
where $C$ is a constant depending only on $H^1$-norm
of $(u_0, w_0)$. This global bound, together with (\ref{decay-l2-est}),
allows us to obtain the global exponential bound for the gradient
of $(u, w)$. In fact, if we set
$$
C_1= \min\left\{\frac{4 \kappa(\gamma-4\kappa)}{\gamma +12 \kappa},\,\frac{C_0}{2}\right\}
$$
and multiply (\ref{H1-decay}) by $e^{C_1 t}$ and integrate in time, we have
\beno
&&\frac{d}{dt}\left(e^{C_1 t} \|\Omega\|^{2}_{L^2(D)}+ e^{C_1 t}\|\nabla w\|^{2}_{L^2(D)}\right) \le C_1 e^{C_1t}\, \|\nabla w\|^{2}_{L^2(D)}\\
&& \qquad  + \,C\,(1+\|u\|_{L^2(D)}^2)(\|u\|_{L^2(D)}^2+\|\Omega\|_{L^2(D)}^2)\, e^{C_1 t}\,\|\nabla w\|_{L^2(D)}^2.
\eeno
Integrating in time and recalling the global bounds
in (\ref{inttt}) and (\ref{ggg}) yields that
$$
e^{C_1 t} \|\Omega(t)\|^{2}_{L^2(D)}+ e^{C_1 t}\|\nabla w(t)\|^{2}_{L^2(D)} \le C(u_0, w_0),
$$
where $C(u_0, w_0)$ depends on $\|u_0\|_{H^1(D)}$ and $\|w_0\|_{H^1(D)}$ only. This completes the proof of Theorem \ref{T3}.
\end{proof}

\vskip .4in
\section*{Acknowledgments}
Q. Jiu was partially supported by National Natural Sciences Foundation of China (No. 11231006 and No. 11671273). J. Wu was partially supported by National Science Foundation grants DMS 1209153
and DMS 1614246, by
the AT\&T Foundation at Oklahoma State University, and by National Natural Sciences Foundation of China (No. 11471103,
a grant awarded to Professor B. Yuan).

\end{document}